\documentclass{siamart250211}

\usepackage{microtype}
\usepackage{amssymb}
\usepackage{bm}
\usepackage{multirow} 

\usepackage{epstopdf}
\ifpdf
\DeclareGraphicsExtensions{.eps,.pdf,.png,.jpg}
\else
\DeclareGraphicsExtensions{.eps}
\fi
\usepackage{xcolor}
\usepackage{tikz}
\usepackage{float}
\usepackage{graphicx}
\usepackage[caption = false]{subfig}
\graphicspath{{fig_file/}} 

\usepackage{algorithmic}
\newsiamremark{remark}{Remark}
\newsiamthm{coro}{Corollary}
\newsiamthm{prop}{Proposition}

\newcommand{\newsiamdefn}[2]{
	\theoremstyle{plain}
	\theoremheaderfont{\normalfont\sc}
	\theorembodyfont{\normalfont}
	\theoremseparator{.}
	\theoremsymbol{}
	\newtheorem{#1}{#2}
}
\newsiamdefn{defn}{Definition}
\newsiamdefn{amp}{Assumption}

\usepackage{amsfonts}
\usepackage{mathtools}
\usepackage{nicefrac}
\usepackage{mathrsfs}
\usepackage{nicematrix}
\usepackage{amsopn}

\DeclareMathOperator{\spanned}{span}

\def\R{\mathbb{R}}

\def\P{\mathcal{P}}

\def\L{\mathcal{L}}

\def\G{\mathcal{G}}

\def\eps{\epsilon}

\allowdisplaybreaks[4]

\headers{An orthogonality-preserving approach for EVPs}{T. Chu, X. Dai, S. Wang, and A. Zhou}

\title{An orthogonality-preserving approach\\ for eigenvalue problems\thanks{Version: \today. 
		\funding{This work was supported by the National Natural Science Foundation of China undergrants 12571446 and 92270206, the Strategic Priority Research Program of the Chinese Academy of Sciences under grant XDB0640000, and the National Key R \& D Program of China under grants 2019YFA0709600 and 2019YFA0709601.}}
}

\author{
	Tianyang Chu\thanks{SKLMS, Academy of Mathematics and Systems Science, Chinese Academy of Sciences, Beijing 100190, China (\email{tchu@lsec.cc.ac.cn}).}
	\and Xiaoying Dai\thanks{SKLMS, Academy of Mathematics and Systems Science, Chinese Academy of Sciences, Beijing 100190, China; and School of Mathematical Sciences, University of Chinese Academy of Sciences, Beijing 100049, China (\email{daixy@lsec.cc.ac.cn}, \email{wangshengyue@amss.ac.cn}, \email{azhou@lsec.cc.ac.cn}).}
	\and Shengyue Wang\footnotemark[3]
	\and Aihui Zhou\footnotemark[3]
}

\begin{document}
	
	\maketitle
	
	\begin{abstract}
		Solving large-scale eigenvalue problems poses a significant challenge due to the computational complexity and limitations on the parallel scalability of the orthogonalization operation, when many eigenpairs are required. In this paper, we propose an intrinsic orthogonality-preserving model, formulated as an evolution equation, and a corresponding numerical method for eigenvalue problems. The proposed approach automatically preserves orthogonality and exhibits energy dissipation during both time evolution and numerical iterations, provided that the initial data are orthogonal—thus offering an accurate and efficient approximation for the large-scale eigenvalue problems with orthogonality constraints. Furthermore, we rigorously prove the convergence of the scheme without the time step size restrictions imposed by the CFL conditions. Numerical experiments not only corroborate the validity of our theoretical analyses but also demonstrate the remarkably high efficiency of the algorithm.
	\end{abstract}
	
	
	
	\section{Introduction}\label{sec: intro}
	
	Eigenvalue problems play a crucial role in many fields of science and engineering, such as quantum mechanics \cite{dirac1929quantum,shankar2012principles}, structural dynamics \cite{bathe1973solution}, and data analysis \cite{Jolliffe2002Principal,mackiewicz1993principal}. In computational chemistry and materials science \cite{le2005computational,leszczynski2012handbook,martin2020electronic}, it is often required to find many eigenvalues and ensure that the corresponding eigenfunctions are mutually orthogonal. Specifically, electronic structures are often modeled by the Hartree-Fock or Kohn-Sham equations \cite{cances2023density,parr1989density,saad2010numerical}, which are nonlinear eigenvalue problems. These models involve computing numerous eigenpairs, a process that, after discretization, is typically reduced to repeatedly solving large-scale linearized eigenvalue problems via self-consistent field (SCF) iterations \cite{cances2023density, chen2014adaptive, chen2013numerical, saad2010numerical}. In each SCF iteration, one must solve for eigenpairs under orthogonality constraints to ensure the eigenfunctions are mutually orthogonal. Specifically, these eigenfunctions correspond to the wavefunctions of electrons in a system and are widely referred to as orbitals.

	However, one of the major challenges in solving these large-scale eigenvalue problems with many eigenpairs is the orthogonalization operation performed to ensure that the eigenfunctions are mutually orthogonal \cite{martin2020electronic,saad2010numerical,Sholl2009}. This orthogonalization step introduces additional computational complexity, which grows dramatically with the problem size, making large-scale calculations computationally challenging. Moreover, its associated communication overhead severely limits the parallel scalability of conventional eigenvalue solvers \cite{golub2013matrix,Saad1992}. These challenges motivate the development of new approaches that mitigate or avoid the frequent, costly eigenfunction orthogonalization steps, enabling far more efficient computation of large-scale eigenvalue problems.
	
	Dai et al. \cite{dai2020,dai2021convergent} proposed an extended gradient flow based Kohn–Sham model whose time evolution intrinsically preserves orthogonality of the orbitals. In their schemes, the Kohn–Sham orbitals evolve according to an $L^2$ gradient flow, and a midpoint (or midpoint-like) time discretization is used to update the solution. This approach completely avoids any explicit orthogonalization step. However, the time steps of these schemes are restricted by Courant-Friedrichs-Lewy (CFL) conditions. Meanwhile, the $L^2$ gradient flow tends to converge slowly for high-frequency components of the orbitals and typically forces the use of small time steps for stability, which greatly limits the size of the time step. The restriction of time step leads to an increase in the number of iteration steps to convergence, thereby resulting in a long convergence time and high computational cost.
	
	In this paper, we develop an intrinsic orthogonality-preserving model and its corresponding numerical method for eigenvalue problems that address the drawbacks of existing methods. Our main contributions are as follows:
	\begin{itemize}
		\item We propose and analyze an orthogonality-preserving model described by the evolution equation (\ref{equ: evolution problem}). For this model, we establish its global well-posedness (Theorem \ref{thm:global well-posedness}) to guarantee solution existence and uniqueness, and further prove two key intrinsic properties of the model: if the initial data are orthogonal, the solutions maintain orthogonality at all times, and the energy decays monotonically over time (Proposition \ref{prop: orthogonality preserving}).
		
		\item We prove the exponential convergence of the solution to the ground state under certain conditions on the initial energy (Theorem \ref{thm: convergence to the ground state}). This ground state is also the solution of the corresponding eigenvalue problem, making our model suitable for solving eigenvalue problems.
		
		\item For the model (\ref{equ: evolution problem}), we propose an explicit time-stepping scheme for the temporal discretization, which is proven to be orthogonality-preserving (Theorem \ref{thm: orthogonality of Un}). This scheme ensures that the orthogonality of the solution is maintained throughout the numerical simulation. We also design an algorithm (Algorithm \ref{alg:Discretization scheme}) that eliminates the need for implicit solving, making each iteration simple and parallel-friendly. 
		
		\item We prove that our orthogonality-preserving iteration scheme produces approximations that converge to the solutions of the corresponding eigenvalue problem with an exponential convergence rate and the energy decreases exponentially under reasonable assumptions (Theorem \ref{thm: exponential convergence of Un by orbtial}). These convergence properties guarantee the accuracy and efficiency of our numerical method.
		
		\item Our numerical analysis is conducted within an infinite-dimensional Hilbert space. This feature naturally avoids any time step size restrictions imposed by CFL conditions-with the corresponding numerical results shown in Table~\ref{table:Mesh-independent time step}. This freedom from CFL constraints offers greater flexibility in implementation (no need for small, restrictive time steps) and highlights the high efficiency of the proposed method.
	\end{itemize}
	It should be pointed out that each orbital in the initial data will evolve independently to its corresponding orbital in the equilibrium state. In contrast to the convergence of the subspace spanned by all orbitals, our approach ensures that the convergence occurs separately for each orbital, a property we refer to as ``orbital-wise" convergence.
	
	The remainder of the paper is organized as follows: In Section \ref{sec: preliminaries}, we introduce the preliminaries, including notation and the problem setting. Section \ref{sec: continuous model} focuses on the proposal of an orthogonality-preserving model and its mathematical analysis. We present the model, prove its orthogonality-preserving property, and establish its global well-posedness. In Section \ref{section:Time discretization}, we propose an orthogonality-preserving iterative method for the model introduced in Section \ref{sec: continuous model}. We conduct a thorough numerical analysis to demonstrate the convergence properties of the proposed scheme and algorithm. Section \ref{section:Numerical experiments} presents several numerical experiments that support our theoretical results. Section \ref{section:Conclusion} concludes the paper by summarizing our contributions and ongoing work. Detailed proofs for the lemmas used throughout the paper are provided in the Appendix.

	\section{Preliminaries}\label{sec: preliminaries}
	
	\subsection{Notation}
	
	Let $\Omega \subset \R^d\ (d\in \mathbb{N}_+)$ be a bounded domain with boundary $\partial\Omega$. Let $H^k(\Omega)$, $k \geqslant 0$, be the conventional Sobolev space, and set 
	\begin{align*}
		L^2(\Omega) = H^0(\Omega) \qquad \text{and} \qquad H_0^1(\Omega)= \{ v \in H^1(\Omega),\,\, v = 0 \,\, \text{on $\partial\Omega$} \}.
	\end{align*}
	We define the inner product $(\cdot,\cdot)$ and norm $\|\cdot\|$ of the space $L^2(\Omega)$ respectively as
	\begin{align*}
		(u,v)=\int_{\Omega} uv \qquad \text{and} \qquad \|u\| =\sqrt{(u,u)}.
	\end{align*}
    
	For $N \in \mathbb{N}_+$ and a Hilbert space $\mathcal{H}$ with inner product $(\cdot,\cdot)_{\mathcal{H}}$, we denote by $\langle U,V \rangle_{\mathcal{H}} \in \R^{N \times N}$ for $U = (u_1,\cdots,u_N), \, V = (v_1,\cdots,v_N)$ the inner matrix
	\begin{align*}
		\big( \langle U,V \rangle_{\mathcal{H}} \big)_{ij} = (u_i,v_j)_{\mathcal{H}}, \qquad i,j \in \{1, \cdots, N\}.
	\end{align*}
	Note that the inner matrix satisfies 
	\begin{align*}
		\langle U, V \rangle_{\mathcal{H}}  = \langle V, U \rangle_{\mathcal{H}}^{\top}, \qquad \forall\, U,V \in \mathcal{H}^N.
	\end{align*}
	For a matrix $Q = (q_{ij})_{i,j = 1}^N \in \R^{N \times N}$ and $U \in \mathcal{H}^N$, $UQ$ denotes an element in $\mathcal{H}^N$ defined by $(UQ)_k = \sum_{i = 1}^N u_iq_{ik}$, and it is easy to see that
	\begin{align*}
		\langle U, VQ \rangle_{\mathcal{H}} = \langle U, V \rangle_{\mathcal{H}}Q \qquad \text{and} \qquad \langle UQ, V \rangle_{\mathcal{H}} = Q^\top\langle U, V \rangle_{\mathcal{H}}.
	\end{align*} 
	With the inner matrix, the inner product and norm of the Hilbert space $\mathcal{H}^N$ can be read as 
	\begin{align*}
		(U,V)_{\mathcal{H}}= \text{tr}(\langle U, V \rangle_{\mathcal{H}}) \qquad \text{and} \qquad \|U\|_{\mathcal{H}} = \sqrt{(U,U)_{\mathcal{H}}}.
	\end{align*}
	
	Let \(\lambda(A)\) denote an eigenvalue of the matrix $A$ (i.e., a scalar satisfying \(A\bm{\alpha} = \lambda \bm{\alpha} \) for some non-zero \(\bm{\alpha}\)). And we define \(\lambda_{\min}(A)\) as the minimum eigenvalue of $A$.
	
	In this paper, we denote by $C$ a generic constant which may be different at different occurrences.

	\subsection{Problem settings}
	
	Let $\mathcal{V}: \Omega \to \R$ be a potential function and we define the following bilinear form on  $H_0^1(\Omega) \times H_0^1(\Omega)$, 
	\begin{align*}
		a(u,v) \stackrel{\Delta}{=} \int_{\Omega}\left(  \nabla u \cdot \nabla v + \mathcal{V}uv \right), \qquad \forall\, u,v \in H_0^1(\Omega).
	\end{align*} 
	Throughout this paper, we make the following assumption on $ \mathcal{V}$: The bilinear form $a(\cdot,\cdot)$ is an inner product on $H_0^1(\Omega)$, and there exist two positive constants $c_1$ and $c_2$ such that  
	\begin{align*}
		c_1 \|\nabla u\|^2 \leqslant a(u,u)\leqslant c_2 \|\nabla u\|^2, \qquad \forall\, u \in H_0^1(\Omega).
	\end{align*} 
	It should be mentioned that our results are also valid for a more general bilinear form $a(\cdot,\cdot)$ that satisfies
	\begin{align*}
		a(u,u)& \geqslant  \frac12 \|\nabla u \|^2 - C \|u\|^2,\qquad \forall\, u \in H_0^1(\Omega),
	\end{align*}
	where $C>0$ is some constant (see Remark 2.9 in \cite{dai2008convergence} for details).
	Consequently, the space $H_0^1(\Omega)$ equipped with the inner product $a(\cdot,\cdot)$ is a Hilbert space, and we use notations $(\cdot,\cdot)_a \stackrel{\Delta}{=} a(\cdot,\cdot)$ and $\|\cdot\|_a  \stackrel{\Delta}{=} \sqrt{(\cdot,\cdot)_a}$ to represent the inner product and the norm, respectively.

	The classical PDE theory \cite{evans2010partial} shows that the following eigenvalue problem: Find $(u,\lambda) \in H_0^1(\Omega) \times \R$ 
	such that the equation
	\begin{align*}
		- \Delta u +  \mathcal{V} u = \lambda u, \qquad \|u\|=1
	\end{align*}
	holds in the weak sense and admits infinitely many eigenvalues $0 < \lambda_1 \leqslant \lambda_2 \leqslant \lambda_3 \leqslant \cdots $. 
     The eigenfunction $u$ corresponding to the smallest eigenvalue $\lambda_1$ satisfies 
	\begin{align*}
		E(u) = \min_{v \in \mathcal{M}} E(v), \qquad \mathcal{M} = \{ v \in H_0^1(\Omega): \, \|v\|=1 \}, 
	\end{align*}
	where $E(v)$ denotes the energy functional defined by
	\begin{align*}
		E(v) \stackrel{\Delta}{=} \frac12 (v,v)_a = \frac12 \int_{\Omega} (|\nabla v|^2 +  \mathcal{V}v^2). 
	\end{align*}
	
	In practical applications, it is often necessary to compute several eigenvalues and their corresponding mutually orthogonal eigenfunctions simultaneously. In other words, we seek $(u_i,\lambda_i) \in H_0^1(\Omega) \times \R \ (i = 1,2,\cdots)$ such that
	\begin{equation}\label{equ: model problem 0}
		- \Delta u_i +  \mathcal{V}u_i = \lambda_i u_i, \qquad (u_i,u_j) = \delta_{ij}
	\end{equation}
     holds in the weak sense.
    In particular, we focus on solving the eigenvalue problem in the weak sense with the $N$ smallest eigenvalues:
	\begin{equation}\label{equ: model problem 1}
		- \Delta U^* +  \mathcal{V} U^* = U^* \Lambda^*, \qquad U^* \in \mathcal{M}^N,
	\end{equation}
	where $\mathcal{M}^N \stackrel{\Delta}{=} \{ U \in [H_0^1(\Omega)]^N:  \langle U, U \rangle = I_N \}$ is the set of orthogonal functions in $[H_0^1(\Omega)]^N$, $\Lambda^* = \text{diag}\left(\lambda_1,\lambda_2, \cdots,\lambda_N\right)$ with $\lambda_N < \lambda_{N+1}$, and the columns of $U^*$ are the corresponding eigenfunctions.
	
	On the other hand, the solution $(U^*,\Lambda^*)$ also satisfies the minimization problem
	\begin{align}\label{equ: model problem 2}
		E(U^*) = \min_{U \in \mathcal{M}^N} E(U),
	\end{align} 
	where the energy functional $E(U)  \stackrel{\Delta}{=} \frac12 (U,U)_a$. Note that the assumption $\lambda_N < \lambda_{N+1}$ guarantees that for any solution $\bar{U} \in \mathcal{M}^N$ of the minimization problem \eqref{equ: model problem 2}, there exists an orthogonal matrix $Q$ such that $\bar{U} = U^*Q$ (see \cite{schneider2009direct}). Then it is natural to consider the equivalence relation $\thicksim$ on $\mathcal{M}^N$: $U \thicksim V$ implies $V = UQ$ for some orthogonal matrix $Q$. Let $[U]$ denote the equivalence class, i.e., $[U] = \{ V \in \mathcal{M}^N: \, U \thicksim V  \}$. For two equivalence classes $[U]$ and $[V]$, the distance between them is defined as
	\begin{align*}
		\| [U] - [V] \|_a = \min\limits_{Q \in \mathcal{O}^N} \| U - VQ\|_a,
	\end{align*}
	where $\mathcal{O}^N$ denotes the set of all $N \times N$ orthogonal matrices.
	
	Hence, it is convenient to solve $[U^*]$ instead of $U^*$. In this sense, every $ \bar{U} \in [U^*]$ satisfies 
	\begin{align*}
		E(\bar{U}) = \min_{U \in \mathcal{M}^N} E(U)
	\end{align*}
	and there exists an orthogonal matrix $Q^*$ such that 
	\begin{align}\label{equ: model problem 3}
		- \Delta \bar{U} +  \mathcal{V} \bar{U} = \bar{U} \Lambda, \qquad \Lambda = \Lambda^* Q^*
	\end{align}
    holds in the weak sense.

	\section{An orthogonality-persevering model}\label{sec: continuous model}
	
	We exploit the analytical structure of the original problem \eqref{equ: model problem 1} and find that the solution of \eqref{equ: model problem 1} can be attained by following a suitable evolution equation whose solutions $U(t)$ preserve orthogonality when initialized with orthogonal data and should approach the minimizer $U^*$ as $t \to \infty$. We refer to this equation as the orthogonality-preserving model.
	
	\subsection{The model}

  Different from the eigenvalue problem (\ref{equ: model problem 1}) and the minimization problem (\ref{equ: model problem 2}), we introduce the following evolution problem: Seek a solution $U(t) \in C^1\big([0,\infty);[H_0^1(\Omega)]^N\big)$ such that
	\begin{equation}\label{equ: evolution problem}
		\frac{\text{d}U}{\text{d}t} = - U \langle \G U,  U \rangle  + \G U\langle U, U \rangle, \qquad U(0) = U^0 \in \mathcal{M}^N,
	\end{equation}
    where the operator $\G =(-\Delta +  \mathcal{V})^{-1} : H^{-1}(\Omega) \to H_0^1(\Omega)$.
    \begin{remark}
        We conclude from \eqref{equ: model problem 3} that the solution $U^*$ of \eqref{equ: model problem 1} satisfies  
	\begin{align*}
		U^* = \G U^* \langle U^*, \G U^* \rangle^{-1} \langle U^*, U^* \rangle.
	\end{align*}
	Due to $\langle U^*, U^* \rangle = I_N$, there holds 
	\begin{align*}
		U^* \langle U^*, \G U^* \rangle  = \G U^*\langle U^*, U^* \rangle.
	\end{align*}
	Thus, $U^*$ can be viewed as the steady state of (\ref{equ: evolution problem}). Consequently, the limit of the solution $U(t)$ of \eqref{equ: evolution problem} can be used to approximate $U^*$, thereby enabling us to obtain the solution to the eigenvalue problem \eqref{equ: model problem 1} by solving the evolution problem (\ref{equ: evolution problem}).
    \end{remark}

	Fixing $U \in [H_0^1(\Omega)]^N$, we define the operator $\mathcal{L}_U: [H_0^1(\Omega)]^N \to [H_0^1(\Omega)]^N$ as
	\begin{equation}\label{definition:operator L_U}
		\mathcal{L}_U V= U \langle \mathcal{G} U, V \rangle - \mathcal{G} U\langle U, V \rangle, \qquad \forall \, V \in [H_0^1(\Omega)]^N.
	\end{equation}
	Therefore, the equation \eqref{equ: evolution problem} can be written as
	\begin{equation}\tag{\ref{equ: evolution problem}$*$}\frac{\mathrm{d}U}{\mathrm{d}t} = - \mathcal{L}_U U, \qquad U(0) = U^0 \in \mathcal{M}^N.
	\end{equation}
	
		We see from \cite{evans2010partial} that the operator $\G$ is well-defined and satisfies 
	\begin{align*}
		\|\G u\|_a \leqslant C \|u\|_{H^{-1}(\Omega)} \leqslant C \|u\|, \qquad \forall\, u \in L^2(\Omega).
	\end{align*}
	Moreover, the operator $\G$ satisfies the following properties: 
	\begin{align*}
		&(\G u,v)  = (\G v ,u ) = (\G u, \G v)_a,\quad \forall\, u,v \in L^2(\Omega), \\
		&(\G u,v)_a  = (\G v ,u )_a  = (u,v),\quad \forall\, u,v \in L^2(\Omega).
	\end{align*}
	These identities imply semi-positive definiteness, as stated in the following lemma. Its proof is provided in Appendix~\ref{proof of lemma:B_n exists}.
	\begin{lemma}\label{lemma:B_n exists}
		For $U \in [L^2(\Omega)]^N$ with $\langle U, U \rangle = I_N$. The matrix 
		\begin{align*}
			\langle  \mathcal{G}U, \mathcal{G} U \rangle -   \langle  \mathcal{G}U,  U \rangle \langle \mathcal{G}U,  U \rangle
		\end{align*}
		is semi-positive definite.
	\end{lemma}
	Furthermore, these identities also establish the skew-symmetry of the operator $\mathcal{L}_U$, formalized in the subsequent lemma.
	\begin{lemma}\label{lem: skew-symmetry}
		If $U,V,W \in [H_0^1(\Omega)]^N$, then 
		\begin{align*}
			\langle V, \mathcal{L}_U W \rangle + \langle \mathcal{L}_U V, W \rangle = 0.
		\end{align*}
		Moreover, for any $U \in \mathcal{M}^N$, there holds $\langle U, \mathcal{L}_U U \rangle = 0$.
	\end{lemma}
    \begin{proof}
    The proof is provided in Appendix~\ref{proof of lemma: skew-symmetry}. 
    \end{proof}
	
	We will see from Proposition \ref{prop: orthogonality preserving} that the operator $\mathcal{L}_U$ defined by \eqref{definition:operator L_U} guarantees that the solution of \eqref{equ: evolution problem} remains orthogonal and exhibits energy dissipation. Namely, \eqref{equ: evolution problem} is an orthogonality-preserving model whenever the initial condition is orthogonal.
	\begin{proposition}\label{prop: orthogonality preserving}
	Suppose that $U(t) \in C^1([0,T); [H_0^1(\Omega)]^N)$ with $T\in (0,+\infty]$ is the solution of (\ref{equ: evolution problem}), then for all $t\in [0, T)$,
		\begin{align*}
			\langle U(t), U(t) \rangle = I_N \qquad \text{and} \qquad \frac{\mathrm{d}E(U(t))}{\mathrm{d}t} \leqslant 0.
		\end{align*}
	\end{proposition}
	\begin{proof}
		First, by Lemma \ref{lem: skew-symmetry}, we have
		\begin{align*}
			\frac{\text{d}}{\text{d}t} \langle U(t), U(t) \rangle & = \langle U(t), \frac{\text{d}U}{\text{d}t} \rangle + \langle \frac{\text{d}U}{\text{d}t}, U(t) \rangle \
			\\& = - \langle U, \mathcal{L}_U U \rangle - \langle \mathcal{L}_U U, U \rangle = 0.
		\end{align*}
		This result implies that $\langle U(t), U(t) \rangle \equiv \langle U^0, U^0 \rangle = I_N$, which means that the orthogonality of $U(t)$ is preserved throughout the evolution.
		
		Next, note that 
		\begin{equation} \label{equ: 1}
			\begin{aligned}
				\langle \G U, U'(t) \rangle_a & = - \langle \G U, U \rangle_a \langle \G U , U \rangle + \langle \G U, \G U \rangle_a \langle U, U \rangle \\
				& = - \langle  U, U \rangle \langle \G U , U \rangle + \langle \G U , U \rangle \langle  U, U \rangle = 0.
			\end{aligned}
		\end{equation}
		We define the operator $\mathcal{T}: [H_0^1(\Omega)]^N \to [H_0^1(\Omega)]^N$ by
        $\mathcal{T}(U) = U - \G U \langle \G U, U \rangle^{-1}\langle U, U \rangle,\ \forall U \in [H_0^1(\Omega)]^N$. We obtain from \eqref{equ: 1} that
		\begin{align*}
			\frac{\mathrm{d}E(U(t))}{\mathrm{d}t} & = \text{tr}(\langle U, U'(t)\rangle_a)  = \text{tr} (\langle \mathcal{T}(U) , U'(t) \rangle_a) \\
			& = - \text{tr} (\langle \mathcal{T}(U) , \mathcal{T}(U)  \rangle_a \langle \G U, U \rangle) \leqslant 0,
		\end{align*}
		which completes the proof, demonstrating that the energy of the system decreases over time.
	\end{proof}
	
	The following subsections provide a mathematical analysis of the model \eqref{equ: evolution problem}, investigating its well-posedness and asymptotic behavior.
	
	\subsection{Well-posedness}
	
	We first consider the local boundedness and local Lipschitz properties of the operator $\L_U$. The proofs for these properties are provided in Appendix~\ref{proof of lem: bound property of L_U} and Appendix~\ref{proof of lem: Local Lipschitz property}, respectively.
	\begin{lemma}\label{lem: bound property of L_U}
		If $U \in [H_0^1(\Omega)]^N$, then
		\begin{align*}
			\|\L_U V \|_a \leqslant C \|U\|_a^2 \|V\|, \qquad \forall\, V \in [H_0^1(\Omega)]^N.
		\end{align*}
	\end{lemma}
	\begin{lemma}\label{lem: Local Lipschitz property}
		If $U,V \in  [H_0^1(\Omega)]^N$ satisfy $\|U\|_a \leqslant M$ and $\|V\|_a \leqslant M$, then
		\begin{align*}
			\| \mathcal{L}_U U - \mathcal{L}_V V \|_a \leqslant C_M \|U- V\|_a,
		\end{align*}
		where $C_M$ is independent of $U,V$.
	\end{lemma}

	With the help of Lemma \ref{lem: Local Lipschitz property} and the classical Picard--Lindel\"of theorem, we have the following local well-posedness of the model \eqref{equ: evolution problem}. 
	\begin{lemma}
		There exists a time $T > 0$ such that the equation \eqref{equ: evolution problem} admits a unique solution satisfying
		\begin{align*}
			U \in C^1([0,T);[H_0^1(\Omega)]^N).
		\end{align*}
	\end{lemma}
	
	The following lemma is a consequence of Proposition \ref{prop: orthogonality preserving}, whose proof is given in Appendix~\ref{proof of lem: uniform bound}.
	\begin{lemma}\label{lem: uniform bound}
		If $U(t)$ is the solution of \eqref{equ: evolution problem}, then the eigenvalues of $\langle \G U(t), U(t) \rangle$ are bounded for any $t\in [0,T)$, that is,
		\begin{align*}
			\lambda(\langle \G U(t), U(t) \rangle) \in [C_1, C_2] ,\qquad \forall\, t \in [0,T).
		\end{align*}
		where $C_1$ and $C_2$ are constants independent of $t$.
	\end{lemma}

	Building on the local well-posedness and utilizing the uniform bound provided by Lemma \ref{lem: uniform bound}, we can extend the solution globally in time. This gives the following global well-posedness for \eqref{equ: evolution problem}, aligning in spirit with \cite[Theorem 3.2]{henning2020sobolev}.
	\begin{theorem}\label{thm:global well-posedness}
		The model \eqref{equ: evolution problem} admits a unique global solution
		\begin{align*}
			U \in C^1([0,\infty); [H_0^1(\Omega)]^N).
		\end{align*}
	\end{theorem}
	\begin{proof}
		Assume that $T$ is finite and maximal in the sense that the problem is no longer well-posed for $t \geqslant T$. Then the energy reduction guarantees that $E_T  \stackrel{\Delta}{=}\lim\limits_{t \to T} E(U(t))$ exists. Let $t_1, t_2 \in [0,T)$ be arbitrary with $t_1 \leqslant t_2$. Then we have 
		\begin{align*}
			\| U(t_2) - U(t_1)\|_a^2 & = \left\| \int_{t_1}^{t_2} U'(t) \ \text{d}t\right\|_a^2 \leqslant \left( \int_{t_1}^{t_2} \|U'(t)\|_a \ \text{d}t \right)^2 \\
			& \leqslant (t_2 - t_1) \int_{t_1}^{t_2} \|U'(t)\|_a^2 \ \text{d}t = (t_2 - t_1) \int_{t_1}^{t_2} \text{tr}(\langle \L_UU,  \L_UU\rangle_a) \ \text{d}t, 
		\end{align*}
		which, combined with the definition of $\mathcal{T}$, leads to
		\begin{equation*}
			\begin{aligned}
				\| U(t_2) - U(t_1)\|_a^2& \leqslant (t_2 - t_1) \int_{t_1}^{t_2} \text{tr}(\langle \G U, U\rangle \langle \mathcal{T}(U) ,  \mathcal{T}(U) \rangle_a\langle \G U, U\rangle) \ \text{d}t \\
				& = (t_2 - t_1) \int_{t_1}^{t_2} \text{tr}(\langle \G U, U\rangle \langle \G U, U\rangle^{\frac12} \langle \mathcal{T}(U) ,  \mathcal{T}(U) \rangle_a\langle \G U, U\rangle^{\frac12}) \ \text{d}t \\
				& \leqslant C (t_2 - t_1) \int_{t_1}^{t_2} \text{tr}( \langle \G U, U\rangle^{\frac12} \langle \mathcal{T}(U) ,  \mathcal{T}(U) \rangle_a\langle \G U, U\rangle^{\frac12}) \ \text{d}t. 
			\end{aligned}
		\end{equation*}
		Consequently, we obtain
		\begin{equation*}
			\begin{aligned}
				\|U(t_2)-U(t_1)\|_a^2	& \leqslant C (t_2 - t_1) \int_{t_1}^{t_2} \text{tr}( \langle \mathcal{T}(U) ,  \mathcal{T}(U) \rangle_a\langle \G U, U\rangle) \ \text{d}t \\
                & =  C (t_2 - t_1) \int_{t_1}^{t_2}\frac{\text{d}E(U(t))}{\text{d}t}\ \text{d}t
				\\& = C (t_2 - t_1) (E(U(t_1)) - E(U(t_2))).
			\end{aligned}
		\end{equation*}
		Hence, there exists a sequence $\{t^n\}$ with $t^n \to T$ and a function $U_T \in [H_0^1(\Omega)]^N$ so that $U(t^n) \rightharpoonup U_T$ weakly in $[H_0^1(\Omega)]^N$. This implies 
		\begin{align*}
			\|U(t^n)\|_a & \leqslant \|U_T\|_a + \| U(t^n) - U_T\|_a  \leqslant \|U_T\|_a + \liminf_{m \to \infty} \| U(t^n) - U(t^m)\|_a \\
			& \leqslant \|U_T\|_a + C \liminf_{m \to \infty} \sqrt{(t^m - t^n) (E(U(t^n)) - E(U(t^m))) } \\
			& \leqslant  \|U_T\|_a + C \sqrt{(T  - t^n) (E(U(t^n)) - E_T) },
		\end{align*}
		which yields
		\begin{align*}
			\limsup_{n \to \infty} \|U(t^n)\|_a \leqslant \|U_T\|_a.
		\end{align*}
		Together with $\|U_T\|_a \leqslant \liminf_{n \to \infty} \|U(t^n)\|_a$, we have $\lim_{n \to \infty } \|U(t^n)\|_a = \|U_T\|_a$, which implies $U(t^n) \to U_T$ strongly in $[H_0^1(\Omega)]^N$. The boundedness of $\|U'(t)\|_a$ (which can be derived from Proposition \ref{prop: orthogonality preserving}) implies $U(t) \to U_T$ as $t \to T$. This contradicts the assumed maximality of $T$.
	\end{proof}
	
	\subsection{Asymptotic behavior}
	
	In this subsection, we investigate the asymptotic behavior of the solution $U(t)$ to the model \eqref{equ: evolution problem}. 
	
	Proposition \ref{prop: orthogonality preserving} implies the following sequential convergence.
	\begin{lemma}\label{lemma: sequence U converge}
		Suppose $U(t)$ is the solution of (\ref{equ: evolution problem}). There exists a sequence $\{t^n\}$ with $t^n \to \infty$ and $\bar{U} \in \mathcal{M}^N$ such that 
		\begin{equation*}
			\|U(t^n) -  \bar{U} \|_a \to 0 \qquad \text{as $n \to \infty$},
		\end{equation*}
		and $\bar{U} \in \mathcal{M}^N$ is a solution of \eqref{equ: model problem 1} as well as a constrained critical point of the energy $E(\cdot)$.
	\end{lemma}
	\begin{proof}
    The proof is provided in Appendix~\ref{proof of lemma: sequence U converge}. 
    \end{proof}
 \begin{remark}
     Combining Lemma \ref{lemma: sequence U converge} and Proposition \ref{prop: orthogonality preserving}, we immediately obtain the following result:
		There holds $E(U(t)) \xrightarrow{t \to \infty} E^*$, where $E^* = E(\bar{U})$ is the energy of some constrained critical point $\bar{U} \in \mathcal{M}^N$ mentioned in Lemma \ref{lemma: sequence U converge}.
        That is, no additional conditions are needed for the model to evolve to a steady state.
 \end{remark}

We define the ground state energy $E_{\text{GS}} = \frac12\sum_{i = 1}^N \lambda_i$ (i.e., $E(U^*)$) and the first excited state energy $E_{\text{ES}} = \frac12 \left( \lambda_{N+1} - \lambda_N + \sum_{i = 1}^N \lambda_i \right)$.  
	Under suitable initial conditions, we can derive the following two lemmas.
    	\begin{lemma}\label{lemma:convergence of equivalence}
		Suppose $U(t)$ is the solution of (\ref{equ: evolution problem}). If the initial value satisfies $E_{\text{GS}} \leqslant E(U^0) < E_{\text{ES}}$, then
		\begin{align*}
			\| [U(t)] - [U^*] \|_a \to 0 \qquad \text{as $t \to \infty$},
		\end{align*}
		where $U^*$ is the solution of \eqref{equ: model problem 2}.
	\end{lemma}
\begin{proof}
The proof is given in Appendix~\ref{proof of lemma:convergence of equivalence}. 
\end{proof}
\begin{lemma}\label{lemma: derivate of U(t) tends to 0}
		Suppose $U(t)$ is the solution of (\ref{equ: evolution problem}). If the initial value satisfies $E_{\text{GS}} \leqslant E(U^0) < E_{\text{ES}}$, then
		\begin{align*}
			\|U'(t)\|_a \to 0 \qquad \text{as $t \to \infty$}.
		\end{align*}
	\end{lemma}
	\begin{proof}
    The proof is deferred to Appendix~\ref{proof of lemma: derivate of U(t) tends to 0}. 
    \end{proof}

	Building on Lemmas \ref{lemma:convergence of equivalence} and \ref{lemma: derivate of U(t) tends to 0}, we establish the exponential convergence of both the solution and the energy in the following theorem.
	\begin{theorem}\label{thm: convergence to the ground state}
		Suppose $U(t)$ is the solution of (\ref{equ: evolution problem}). If the initial value satisfies $E_{\text{GS}} \leqslant E(U^0) < E_{\text{ES}}$, then for any small constant $\epsilon\in (0, \frac{1}{\lambda_N} - \frac{1}{\lambda_{N+1}})$ , there exist a finite time $T_{\epsilon}>0$ and a constant $C_{\epsilon}>0$ such that 
		\begin{align}\label{equ: energy exponential convergence}
			E(U(t)) - E_{\text{GS}} \leqslant C_{\epsilon} \exp\left( -  2\left( \frac{1}{\lambda_N} - \frac{1}{\lambda_{N+1}} - \epsilon \right) t\right), \quad \forall t \geqslant T_{\epsilon},
		\end{align}
		and there exists a $Q^*\in \mathcal{O}^N$ dependent on $t$ such that
		\begin{align}\label{equ: function exponential convergence}
			\| U(t) - U^*Q^*\|_a \leqslant C_{\epsilon} \exp\left( -  \left( \frac{1}{\lambda_N} - \frac{1}{\lambda_{N+1}} - \epsilon \right) t\right), \quad \forall t \geqslant T_{\epsilon},
		\end{align}
        where $U^*$ is the solution of \eqref{equ: model problem 2}.
	\end{theorem}
	\begin{proof}
		Denote $g(t) = \frac12 \|U'(t)\|_a^2$. Then,
		\begin{align*}
			g'(t) & = \text{tr}(\langle U'(t), U''(t) \rangle_a)= \text{tr}\Big(\left\langle U'(t), \big( -\mathcal{T}(U)  \langle \G U, U \rangle \big)' \right\rangle_a \Big) \\
			& = \underbrace{ - \text{tr}\Big(\left\langle U'(t), \big(\mathcal{T}(U)\big)'  \rangle_a\right\langle \G U, U \rangle \Big)}_{=: I_1} \underbrace{- \text{tr}\Big(\langle U'(t), \mathcal{T}(U)   \rangle_a \big(\langle \G U, U \rangle\big)'\Big)}_{=:I_2}.
		\end{align*}
		In the following, estimates for $I_1$ and $I_2$ will be provided. 
        
        Since the derivative of $\mathcal{T}(U)$ satisfies
		\begin{align*}
			\big(\mathcal{T}(U)\big)' = U' - \G U' \langle \G U, U \rangle^{-1} \langle U, U \rangle - \G U\big(\langle \G U, U \rangle^{-1} \langle U, U \rangle\big)',
		\end{align*}
		and by leveraging the key relation (\ref{equ: 1}), we further obtain
		\begin{align*}
			I_1 = - \text{tr}(\langle U'(t), U'(t)  \rangle_a\langle \G U, U \rangle  ) + \|U'\|^2.
		\end{align*}
		
		Note that 
		\begin{align*}
			\text{tr}(\langle U'(t), U'(t)  \rangle_a\langle \G U, U \rangle  ) \geqslant \lambda_{\text{min}}(\langle \G U, U \rangle) \|U'(t)\|_a^2
		\end{align*}
		and for all $t \geqslant 0$, there exists a $Q(t) \in \mathcal{O}^N$ such that
		\begin{align*}
			Q(t)^{\top} \langle \G U, U \rangle Q(t) \to \langle \G U^{*}, U^* \rangle \qquad \text{as $t \to \infty$}.
		\end{align*} 
		Since $U^*$ is the ground state, we have 
		\begin{align*}
			\lambda_{\text{min}}(\langle \G U^*, U^* \rangle) = \frac{1}{\lambda_N}.
		\end{align*}
		Furthermore, for all $\epsilon \in (0, \frac{1}{\lambda_N} - \frac{1}{\lambda_{N+1}})$ there exists a time $T_{\epsilon}$ such that for $t \geqslant T_{\epsilon}$,
		\begin{align*}
			\lambda_{\text{min}}(\langle \G U, U \rangle) \geqslant \frac{1}{\lambda_N} - \epsilon.
		\end{align*}
		
		To estimate $\|U'\|_a^2$, denote $C_{\text{inf}} \stackrel{\Delta}{=} \liminf_{t \to \infty} \frac{\|U'\|_a^2}{\|U'\|^2}$.
		Then, there exists a sequence $\{t^n\}$ such that 
		\begin{align*}
			C_{\text{inf}} = \lim_{n \to \infty}\frac{\|U'(t^n)\|_a^2}{\|U'(t^n)\|^2}.
		\end{align*}
		Let $Z^n  \stackrel{\Delta}{=} U'(t^n)/\|U'(t^n)\|$. Since $\{Z^n\}$ is a bounded sequence in $[H_0^1(\Omega)]^N$, there exists $\hat{Z} \in [H_0^1(\Omega)]^N$ such that $Z^n \rightharpoonup \hat{Z}$ weakly in $[H_0^1(\Omega)]^N$. Using the fact that $U(t^n)Q(t^n) \to U^*$ strongly in $[H_0^1(\Omega)]^N$ and $\langle U'(t^n), U(t^n) \rangle = 0$, we obtain that $\langle \hat{Z}, U^* \rangle = 0$. This implies that for $i = 1, \cdots, N$,
		\begin{align*}
			\| [\hat{Z}]_i\|_a^2 \geqslant \lambda_{N+1} \| [\hat{Z}]_i\|^2,
		\end{align*}
		which yields $\|\hat{Z}\|_a^2 \geqslant \lambda_{N+1} \|\hat{Z}\|^2 = \lambda_{N+1}$. The weak convergence of $\left\{Z^n\right\}$ then implies
		\begin{align*}
			C_{\text{inf}} = \lim_{n \to \infty} \|Z^n\|_a^2 \geqslant \|\hat{Z}\|_a^2 \geqslant \lambda_{N+1}.
		\end{align*}
		
		Therefore, for $t \geqslant T_{\epsilon}$, $I_1$ can be estimated as
		\begin{align*}
			I_1 & \leqslant - \left( \frac{1}{\lambda_N} - \epsilon \right) \|U'(t)\|_a^2 + \frac{1}{\lambda_{N+1}}  \|U'(t)\|_a^2 
		  = - 2 \left( \frac{1}{\lambda_N} - \frac{1}{\lambda_{N+1}} - \epsilon \right) g(t).
		\end{align*}
		
		For $I_2$, we first note that $\|\mathcal{T}(U) \|_a = \|U'\langle \G U, U \rangle^{-1}\|_a \to 0$. Furthermore, since
		\begin{align*}
			[\langle \G U, U \rangle]' & = \langle U', \G U \rangle + \langle \G U,  U' \rangle,
		\end{align*}
		it follows that
		\begin{align*}
			|[\langle \G U, U \rangle]'| \leqslant C \|U'\|_a\|U\| \leqslant C \|U'\|_a.
		\end{align*}
		Consequently, for $t \geqslant T_{\epsilon}$, $I_2$ can be estimated as
		\begin{align*}
			|I_2| \leqslant C \|U'\|_a \|\mathcal{T}(U) \|_a \|U'\|_a \leqslant 2 \epsilon g(t).
		\end{align*}
		
		Combining the estimates of $I_1$ and $I_2$ yields
		\begin{align*}
			g'(t) \leqslant - 2 \left( \frac{1}{\lambda_N} - \frac{1}{\lambda_{N+1}} - \epsilon \right)g(t), \quad \forall t \geqslant T_{\epsilon}.
		\end{align*}
		 Applying the Gronwall's inequality leads to
		\begin{equation}\label{equ: estimate of g}
			g(t) \leqslant C_{\epsilon} \exp\left( -  2\left( \frac{1}{\lambda_N} - \frac{1}{\lambda_{N+1}} - \epsilon \right) t\right), \quad  \forall t \geqslant T_{\epsilon},
		\end{equation}
        where $C_{\epsilon}>0$ is a constant dependent on $\epsilon$.
		
		Finally, \eqref{equ: energy exponential convergence} directly follows from \eqref{equ: estimate of g} and Proposition \ref{prop: orthogonality preserving}, while \eqref{equ: function exponential convergence} directly follows from \eqref{equ: estimate of g} and Lemma \ref{lemma: sequence U converge}.
	\end{proof}
	
	\section{Time discretization}\label{section:Time discretization}
	
	In this section, we propose an orthogonality-preserving numerical method for the model \eqref{equ: evolution problem}. The proposed scheme is designed to simulate the behavior of continuous solutions, and is proven to preserve orthogonality of the solution throughout the time evolution.
	
	\subsection{Numerical scheme}
	
	Let $\left\{t_n: n=0,1,2 \cdots\right\} \subset[0,+\infty)$ be discrete points such that
	$$
	0=t_0<t_1<t_2<\cdots<t_n<\cdots,
	$$
	and $\lim _{n \rightarrow+\infty} t_n=+\infty$. Set
	$$
	\tau_n=t_{n+1}-t_n,
	$$
	and consider the following scheme: Given $U^n$ with $\langle U^n,U^n \rangle  = I_N$, find $U^{n+1}$ such that
	\begin{equation}
		\label{equ: numerical scheme}
		\begin{aligned}
			\frac{U^{n+1} - U^n}{ \tau_n }  = - \mathcal{L}_{U^n} \frac{U^{n+1} + U^n}{2},
		\end{aligned}
	\end{equation}
	where $U^{n + \frac12} = (U^{n+1} + U^n)/2$.
	
	We denote $A^n =  \left\langle  \mathcal{G} U^n, U^{n + \frac12} \right\rangle$ and $B^n  = \left\langle U^n , U^{n + \frac12} \right\rangle$. Thus, $U^{n+1}$ can be obtained by 
	\begin{align}\tag{\ref{equ: numerical scheme}$*$}	
		U^{n+1} = U^n - \tau_n U^n A^n + \tau_n \mathcal{G} U^n B^n.
	\end{align}

	The following theorem gives the well-posedness of the numerical scheme \eqref{equ: numerical scheme}.
	\begin{theorem}\label{thm: orthogonality of Un}
		The numerical scheme \eqref{equ: numerical scheme} is well-posed, i.e., $U^{n+1}$ is well-defined. Moreover, 
		\begin{align*}
			\langle U^{n+1}, U^{n+1} \rangle = \langle U^{n}, U^{n} \rangle = I_N.
		\end{align*}
	\end{theorem}
	\begin{proof}
		The existence of $B_n$ depends on the invertibility of the matrix 
		$$
		I_N + \frac{\tau_n^2}{4} \langle  \mathcal{G}U^n, \mathcal{G} U^n \rangle -  \frac{\tau_n^2}{4} \langle  \mathcal{G}U^n,  U^n \rangle \langle \mathcal{G}U^n,  U^n \rangle.
		$$
		From Lemma \ref{lemma:B_n exists}, we directly establish the well-posedness of the numerical scheme \eqref{equ: numerical scheme}.
		
		It remains to prove that $\langle U^{n+1}, U^{n+1} \rangle = I_N$. Note that
		\begin{align*}
			\langle U^{n+1} + U^{n}, U^{n+1} - U^{n} \rangle = \langle U^{n+1}, U^{n+1} \rangle  - \langle U^{n+1}, U^{n} \rangle + \langle U^{n}, U^{n+1} \rangle - \langle U^{n}, U^{n} \rangle
		\end{align*}
		is skew-symmetric. Therefore, $\langle U^{n+1}, U^{n+1} \rangle - \langle U^{n}, U^{n} \rangle$ is also skew-symmetric. On the other hand, $\langle U^{n+1}, U^{n+1} \rangle - \langle U^{n}, U^{n} \rangle$ is symmetric. Consequently, $\langle U^{n+1}, U^{n+1} \rangle - \langle U^{n}, U^{n} \rangle = 0$, which completes the proof.
	\end{proof}

	Rewrite (\ref{equ: numerical scheme}) and $U^{n + \frac12} = (U^{n+1} + U^n)/2$ into two equations:
	\begin{equation*}
		\left\{
		\begin{aligned}
			&\frac{U^{n+\frac12} - U^n}{ \tau_n/2 } = - U^n A^n+ \mathcal{G} U^nB^n,\\
			&\frac{ U^{n+1}-U^{n+\frac12}}{ \tau_n/2 } = - U^n A^n+ \mathcal{G} U^nB^n.
		\end{aligned}
		\right.
	\end{equation*}
	From the first equation, we obtain
	\begin{align*}
		U^{n+\frac12} = U^n - \frac{\tau_n}{2} U^n A^n + \frac{\tau_n}{2} \mathcal{G} U^n B^n.
	\end{align*}
	Substituting this expression into the definitions of $A^n$ and $B^n$ yields
	\begin{align*}
		A^n & = \langle \mathcal{G}U^n, U^n \rangle - \frac{\tau_n}{2}\langle \mathcal{G}U^n, U^n \rangle A^n + \frac{\tau_n}{2} \langle \mathcal{G}U^n, \mathcal{G} U^n \rangle B^n, \\
		B^n & = \langle U^n, U^n \rangle - \frac{\tau_n}{2}\langle U^n, U^n \rangle A^n + \frac{\tau_n}{2} \langle \mathcal{G}U^n, U^n \rangle B^n.
	\end{align*}
	Rearranging these equations leads to
	\begin{align*}
		\frac{\tau_n}{2} A^n & = I_N - B^n + \frac{\tau_n}{2} \langle \mathcal{G}U^n, U^n \rangle B^n, \\
		I_N & = \left( I_N + \frac{\tau_n^2}{4} \langle \mathcal{G}U^n, \mathcal{G} U^n \rangle - \frac{\tau_n^2}{4} \langle \mathcal{G}U^n, U^n \rangle \langle \mathcal{G}U^n, U^n \rangle \right) B^n.
	\end{align*}
	
	Therefore, $U^{n+1}$ can be produced through the following four sub-steps.
	\begin{itemize}
		\item \textbf{Step 1}: Compute $N$ source problems independently 
		\begin{align*}
			( - \Delta + \mathcal{V} ) \mathcal{G} U^n = U^n.
		\end{align*}
		\item \textbf{Step 2}: Compute $B^n$ from 
		\begin{align*}
			I_N  = \left( I_N + \frac{\tau_n^2}{4} \langle  \mathcal{G}U^n, \mathcal{G} U^n \rangle -  \frac{\tau_n^2}{4} \langle  \mathcal{G}U^n,  U^n \rangle \langle \mathcal{G}U^n,  U^n \rangle \right) B^n.
		\end{align*}
		\item \textbf{Step 3}: Compute $A^n$ by
		\begin{align*}
			\frac{\tau_n}{2} A^n  = I_N - B^n +  \frac{\tau_n}{2} \langle \mathcal{G}U^n, U^n \rangle B^n.
		\end{align*}
		\item \textbf{Step 4}: Update the iterate
		\begin{align*}
			U^{n+1} = U^n - \tau_n U^n A^n + \tau_n \mathcal{G} U^n B^n.
		\end{align*}
	\end{itemize}
	
	The complete iteration process is summarized in Algorithm \ref{alg:Discretization scheme}.
	\begin{algorithm}[H]
		\label{alg:Discretization scheme}
		\caption{}
		\begin{algorithmic}[1]
			\STATE Given tolerance $\epsilon>0$, bounds
			$\tau_{\min},\tau_{\max}>0$, and initial data
			$U^{0}\in\mathcal{M}^{N}$; set $n=0$ and
			$\mathrm{err}_{E}^0=|E(U^{0})|$;
			\WHILE{$\mathrm{err}_{E}^n>\epsilon$}
			\STATE Choose a time step $\tau_n\in[\tau_{\min},\tau_{\max}]$;
			\STATE Execute \textbf{Step 1 $\sim$ Step 4} to obtain $U^{n+1}$;
			\STATE Compute
			$\mathrm{err}_{E}^n=
			|E(U^{n+1})-E(U^n)|\ /\ |E(U^n)|$,
			and set $n= n+1$;
			\ENDWHILE
		\end{algorithmic}
	\end{algorithm}
   This algorithm avoids the need for implicit solves, rendering each iteration straightforward and parallelizable. Additionally, the subsequent analysis in the infinite-dimensional space demonstrates that the time step is not subject to any CFL conditions.
	
	\subsection{Convergence}
	
	In this subsection, we will show the convergence of the numerical solutions. 
	
	We begin by establishing the energy decay property.
	\begin{theorem}\label{thm: Energy E(Un) decay}
		For $\tau_n \leqslant \tau^*$, here $\tau^*$ only depends on $E(U^0)$, the numerical scheme \eqref{equ: numerical scheme} is energy dissipative, that is, for all $n \in \mathbb{N}_+$,
		\begin{align*}
			E(U^{n+1}) \leqslant E(U^n) \leqslant E(U^0).
		\end{align*}
		In particular,
		\begin{align*}
			E(U^n) - E(U^{n+1}) \geqslant C \tau_n \| \mathcal{L}_{U^n}U^{n}\|_a^2,
		\end{align*}
		where $C>0$ is a constant independent of $n\in \mathbb{N}_+$.
	\end{theorem}
	\begin{proof}
		First, we observe that the conclusion holds for $n = 0$. Now, suppose the conclusion is true for $n$. We will show that it also holds for $n+1$.
		
		Begin by noting that
		\begin{align*}
			2 \big(E(U^{n}) - E(U^{n+1})\big) & = \text{tr}(\langle U^n, U^n \rangle_a) - \text{tr}(\langle U^{n+1}, U^{n+1} \rangle_a) \\
			& = \text{tr}(\langle U^n - U^{n+1}, U^n + U^{n+1} \rangle_a ).
		\end{align*}
		This can be rewritten as 
		\begin{align*}
			E(U^{n}) - E(U^{n+1}) & = \text{tr}(\langle U^n - U^{n+1}, U^{n+\frac12} \rangle_a )  = \tau_n \text{tr}(\langle \mathcal{L}_{U^n}U^{n+\frac12}, U^{n+\frac12} \rangle_a ) \\
			& = \tau_n \text{tr}(\langle \mathcal{L}_{U^n}U^{n+\frac12} - \mathcal{L}_{U^n}U^{n} + \mathcal{L}_{U^n}U^{n}, U^{n+\frac12}  - U^n + U^n\rangle_a ). 
		\end{align*}
		As a result, we can divide $E(U^n)-E(U^{n+1})$ as follows
		\begin{equation*}
			\begin{aligned}
				E(U^n)-E(U^{n+1})& = \tau_n \underbrace{\text{tr}(\langle \mathcal{L}_{U^n}U^{n+\frac12} - \mathcal{L}_{U^n}U^{n} , U^{n+\frac12}  - U^n \rangle_a )}_{=:I_1} \\
				& \quad + \tau_n \underbrace{ \text{tr}(\langle \mathcal{L}_{U^n}U^{n+\frac12} - \mathcal{L}_{U^n}U^{n} ,  U^n\rangle_a )}_{=:I_2} \\
				& \quad + \tau_n \underbrace{ \text{tr}( \langle\mathcal{L}_{U^n}U^{n}, U^{n+\frac12}  - U^n \rangle_a )}_{=:I_3}  + \tau_n \underbrace{\text{tr}(\langle  \mathcal{L}_{U^n}U^{n},   U^n\rangle_a )}_{=:I_4}.
			\end{aligned}
		\end{equation*}
		Next, we estimate each of the terms $I_1, I_2, I_3,$ and $I_4$. First, observe that
		\begin{align*}
			\| \mathcal{L}_{U^n}U^{n+\frac12} \|_a & \leqslant \| \mathcal{L}_{U^n}U^{n+\frac12} - \mathcal{L}_{U^n}U^{n} \|_a + \| \mathcal{L}_{U^n}U^{n} \|_a \\
			& \leqslant C \| U^{n+\frac12} - U^n \|_a +  \| \mathcal{L}_{U^n}U^{n} \|_a \\
			& \leqslant C \tau_n \| \mathcal{L}_{U^n}U^{n+\frac12}\|_a +  \| \mathcal{L}_{U^n}U^{n} \|_a,
		\end{align*}
		which implies $\| \mathcal{L}_{U^n}U^{n+\frac12} \|_a \leqslant C \| \mathcal{L}_{U^n}U^{n} \|_a$. Therefore, for $I_1$, it follows that
		\begin{align*}
			|I_1| & \leqslant C \| \mathcal{L}_{U^n}U^{n+\frac12} - \mathcal{L}_{U^n}U^{n} \|_a \| U^{n+\frac12} - U^n \|_a   \leqslant C \| U^{n+\frac12} - U^n \|_a^2 \\
			& \leqslant C \tau_n^2 \| \mathcal{L}_{U^n}U^{n+\frac12}\|_a^2 \leqslant C \tau_n^2 \| \mathcal{L}_{U^n}U^{n}\|_a^2.
		\end{align*}
		Similarly, for $I_3$, we obtain 
		\begin{align*}
			|I_3| & \leqslant C \| \mathcal{L}_{U^n}U^{n}\|_a \| U^{n+\frac12} - U^n \|_a  \leqslant C \tau_n \| \mathcal{L}_{U^n}U^{n}\|_a^2.
		\end{align*}
		And $I_4$ is given as
		\begin{align*}
			I_4 & = \text{tr}(\langle  \mathcal{L}_{U^n}U^{n},   U^n - \G U^n \langle \G U^n, U^n \rangle^{-1} \rangle_a + \langle  \mathcal{L}_{U^n}U^{n},   \G U^n \langle \G U^n, U^n \rangle^{-1} \rangle_a ) \\
			& = \text{tr}(\langle  \mathcal{L}_{U^n}U^{n},   U^n - \G U^n \langle \G U^n, U^n \rangle^{-1} \rangle_a  ) \\
			& = \text{tr}(\langle  \mathcal{L}_{U^n}U^{n},   \mathcal{L}_{U^n}U^{n}  \rangle_a  \langle \G U^n, U^n \rangle^{-1}).
		\end{align*}
		To estimate $I_2$, we decompose it as
		\begin{align*}
			I_2 & = \text{tr}(\langle \mathcal{L}_{U^n}U^{n+\frac12} - \mathcal{L}_{U^n}U^{n} ,  U^n\rangle_a ) \\
			& = \underbrace{\text{tr}(\langle \mathcal{L}_{U^n}U^{n+\frac12} - \mathcal{L}_{U^n}U^{n} ,  U^n - \G U^n \langle \G U^n, U^n \rangle^{-1} \rangle_a )}_{=:E_1} \\
			& \quad + \underbrace{\text{tr}(\langle \mathcal{L}_{U^n}U^{n+\frac12} - \mathcal{L}_{U^n}U^{n} ,   \G U^n \langle \G U^n, U^n \rangle^{-1} \rangle_a )}_{=: E_2}.
		\end{align*}
		Note that 
		\begin{align*}
			\langle  \mathcal{L}_{U^n}U^{n+1}, U^n \rangle = \langle U^{n+1}, \G U^n \rangle I_N - \langle U^{n+1}, U^n \rangle \langle \G U^n , U^n \rangle
		\end{align*}
		and 
		\begin{align*}
			\langle  U^{n+1}, \mathcal{L}_{U^n} U^n \rangle = \langle U^{n+1},  U^n \rangle \langle \G U^n , U^n \rangle - \langle U^{n+1}, \G U^n \rangle I_N ,
		\end{align*}
		which implies
		\begin{align*}
			\langle  \mathcal{L}_{U^n}U^{n+1}, U^n \rangle = - \langle  U^{n+1}, \mathcal{L}_{U^n} U^n \rangle.
		\end{align*}
		Then, for $E_1$, we have 
		\begin{align*}
			|E_1| & \leqslant C \| \mathcal{L}_{U^n}U^{n+\frac12} - \mathcal{L}_{U^n}U^{n}\|_a \| \mathcal{L}_{U^n}U^{n}\|_a \\
			& \leqslant \| U^{n+\frac12} - U^{n}\|_a\| \mathcal{L}_{U^n}U^{n}\|_a \\
			& \leqslant C \tau_n\| \mathcal{L}_{U^n}U^{n+\frac12}\|_a \| \mathcal{L}_{U^n}U^{n}\|_a \leqslant C \tau_n \| \mathcal{L}_{U^n}U^{n}\|_a^2,
		\end{align*}
		and for $E_2$, we obtain 
		\begin{align*}
			E_2 & = \text{tr}(\langle \mathcal{L}_{U^n}U^{n+\frac12} - \mathcal{L}_{U^n}U^{n} ,   \G U^n \langle \G U^n, U^n \rangle^{-1} \rangle_a ) \\
			& = \text{tr}(\langle \mathcal{L}_{U^n}U^{n+\frac12} - \mathcal{L}_{U^n}U^{n} ,    U^n  \rangle\langle \G U^n, U^n \rangle^{-1} ) \\
			& = \text{tr}(-\langle U^{n+\frac12} - U^{n} ,    \mathcal{L}_{U^n} U^n  \rangle\langle \G U^n, U^n \rangle^{-1} ),
		\end{align*}
		which implies 
		\begin{align*}
			|E_2| & \leqslant C \|U^{n+\frac12} - U^{n}\|_a  \| \mathcal{L}_{U^n}U^{n}\|_a \\
			& \leqslant C \tau_n \| \mathcal{L}_{U^n}U^{n+\frac12}\|_a \| \mathcal{L}_{U^n}U^{n}\|_a \leqslant C \tau_n \| \mathcal{L}_{U^n}U^{n}\|_a^2.
		\end{align*}
		Therefore, we have $I_1 \leqslant C \tau_n \| \mathcal{L}_{U^n}U^{n}\|_a^2$.
        
		Combining the estimates for $I_1, I_2, I_3$ and $I_4$, we conclude
		\begin{align*}
			E(U^n) - E(U^{n+1}) \geqslant \tau_n \text{tr}(\langle  \mathcal{L}_{U^n}U^{n},   \mathcal{L}_{U^n}U^{n}  \rangle_a  \langle \G U^n, U^n \rangle^{-1}) - C \tau_n^2 \| \mathcal{L}_{U^n}U^{n}\|_a^2.
		\end{align*}
		Hence, for $\tau_n \leqslant \tau^*$, where $\tau^*$ depends only on $E(U^0)$, there holds
		\begin{align*}
			E(U^n) - E(U^{n+1}) \geqslant C \tau_n \| \mathcal{L}_{U^n}U^{n}\|_a^2.
		\end{align*}
		Consequently,
		\begin{align*}
			E(U^{n+1}) \leqslant E(U^n) \leqslant E(U^0),
		\end{align*}
		which completes the proof by induction.
	\end{proof}
	
	Let us assume that the time step satisfies: 
	\begin{align*}
		\tau_{\text{min}} \leqslant \tau_n \leqslant \tau_{\text{max}},\qquad \forall n \geqslant 0,
	\end{align*}
	where $	0<\tau_{\text{min}} \leqslant \tau_{\text{max}}\leqslant \tau^*$.
	
	\begin{lemma}\label{lem: converge to barU}
		Any weak limit $\bar{U}$ of $\{U^n\}$ is a solution of \eqref{equ: model problem 0}. Moreover, there exists a subsequence of $\{U^n\}$ such that $\bar{U}$ is its strong limit. 
	\end{lemma}
	\begin{proof}
    The proof has been deferred to Appendix~\ref{proof of lem: converge to barU}. 
    \end{proof}

	\begin{theorem}\label{thm: convergence of Un}
		If the initial value $U^0$ satisfies $E_{\text{GS}} \leqslant E(U^0) < E_{\text{ES}}$, then
		\begin{align*}
		    E(U^n)-E_{\text{GS}} &\rightarrow 0 \quad \text{ as } n \rightarrow \infty,\\
		\left\|[U^n] - [U^*] \right\|_a &\rightarrow 0\quad \text{ as } n \rightarrow \infty.
		\end{align*}
	\end{theorem}
	\begin{proof}
		By the conclusion of Lemma \ref{lem: converge to barU} and the energy decay property in Theorem \ref{thm: Energy E(Un) decay}, we obtain that $E(U^n) \to E_{\text{GS}}$, which proves the first claim.
		
		To prove the second claim, we adopt a proof by contradiction. Suppose there exists a sequence $\{U^n\}$ such that $\|U^n - U^* Q_n\|_a \geqslant \epsilon$ for some $\epsilon > 0$. For this sequence, we have $U^n \rightharpoonup \bar{U}$ for some $\bar{U} \in [H_0^1(\Omega)]^N$ with $\langle \bar{U},\bar{U} \rangle = I_N$. Then, by the lower semi-continuity of $E$, it follows that
		\begin{equation*}
			E(\bar{U}) \leqslant \liminf_{n \to \infty} E(U^n) = E_{\text{GS}}.
		\end{equation*}
		This implies that $\bar{U}$ is one of the ground states, i.e., $\bar{U}\in [U^*]$, and $\|U^n\|_a \to \|\bar{U}\|_a$. Together with the weak convergence, we obtain $\|U^n - \bar{U}\|_a \to 0$, which contradicts the assumption that $\|U^n - U^* Q_n\|_a \geqslant \epsilon$. Therefore, the second claim holds.
	\end{proof}
	
	\subsection{Convergence rate}
	
	In the preceding subsection, we established the convergence of the numerical solutions generated by our scheme (\ref{equ: numerical scheme}). In this subsection, we further demonstrate that the numerical solutions exhibit the property of exponential convergence. Moreover, we find a fact that is consistent with the theoretical result in the model (\ref{equ: evolution problem}): the convergence of the columns of the numerical solution is \emph{orbital-wise}. This means that each column of the numerical solution, representing an individual orbital, converges independently to its corresponding orbital in the ground state.

    For $ U = (u_1, u_2, \cdots,u_N) \in [H_0^1(\Omega)]^N$, we define $\spanned(U)$ as the smallest subspace of $H_0^1(\Omega) $ containing all columns of $U$, given by
    $$ \spanned(U)= \left\{ \sum_{i=1}^{N}\alpha_i u_i\,|\ \alpha_i \in \mathbb{R},\, i =1,2,\cdots,N \right\}.$$
	Let $\mathcal{P}: H_0^1(\Omega) \to \text{span}(U^*)$ be the $L^2$-projection and $\mathcal{P}_a: H_0^1(\Omega) \to \text{span}(U^*)$ be the $H_0^1$-projection, respectively, with
	$$
	\P U = (\P u_1, \cdots, \P u_N) \quad\text{and} \quad
	\P_a U = (\P_a u_1, \cdots, \P_a u_N).
	$$ 
	With these notations, we present the following result, and its proof is provided in Appendix \ref{proof of lem: same projection}.
	\begin{lemma}\label{lem: same projection}
		For any $U \in [H_0^1(\Omega)]^N$, there holds
		\begin{align*}
			\P U = \P_a U.
		\end{align*}
	\end{lemma}
	
	We denote $\P_{\bot}U = U- \P U$.
	By the conclusion of the above lemma, we have
	\begin{align*}
		\langle \P U, \P_{\bot}U \rangle = \langle \P U, \P_{\bot}U \rangle_a = 0,
	\end{align*}
	which leads to the following lemma, with its proof given in Appendix \ref{proof of lemma: PGU=GPU}.
	\begin{lemma}\label{lemma: PGU=GPU}
		For any $U \in [H_0^1(\Omega)]^N$, there holds
		\begin{align*}
			\P_{\bot} (\G U) = \G (\P_{\bot}U).
		\end{align*}
	\end{lemma}

	\begin{lemma}\label{lem: exponential convergence of numerical scheme}
		Suppose the initial value $U^0$ satisfies $E_{\text{GS}} \leqslant E(U^0) < E_{\text{ES}}$. If $\tau_n \in [\tau_{\text{min}}, \tau_{\text{max}}]$, with $\tau_{\text{max}}$ sufficiently small, then there exist $\omega \in (0,1)$ and $n_0\in \mathbb{N}_+$ such that 
		\begin{align*}
			\|\P_{\bot}U^{n+1}\|_a \leqslant \omega \|\P_{\bot}U^n\|_a, \qquad \forall n \geqslant n_0.
		\end{align*}
	\end{lemma}
	\begin{proof}
    The proof is provided in Appendix~\ref{proof of lem: exponential convergence of numerical scheme}. \end{proof}

	Now, we introduce the notion of the distance between two spaces \cite{babuska1991eigenvalue,chatelin2011spectral}. For two finite-dimensional subspaces $M$ and $N$ of a Hilbert space $X$ with $\operatorname{dim}M =\operatorname{dim}N$ , we define
	\begin{align*}
		\delta_X(M,N) = \sup_{x \in M, \, \|x\|_X = 1} \text{dist}_X(x,N), \qquad \text{dist}_X(x,N) \stackrel{\Delta}{=} \inf_{y \in N} \|x - y\|_X
	\end{align*}
	as the distance between the spaces $M$ and $N$. Based on this definition, we denote the space distance in $H_0^1(\Omega)$ and $L^2(\Omega)$ as
	\begin{align*}
		\delta_{H_0^1}(M,N) = \sup_{x \in M,\, \|x\|_a = 1} \text{dist}_{H_0^1}(x,N), \quad \delta_{L^2}(M,N) = \sup_{x \in M,\, \|x\| = 1} \text{dist}_{L^2}(x,N) . 
	\end{align*}
	
	\begin{theorem}\label{thm: exponential convergence of space distance}
		Suppose the initial value $U^0$ satisfies $E_{\text{GS}} \leqslant E(U^0) < E_{\text{ES}}$. If $\tau_n \in [\tau_{\text{min}}, \tau_{\text{max}}]$, with $\tau_{\text{max}}$ sufficiently small, then there exist $C>0, c >0$ such that 
		\begin{align}
			{\delta}_{H_0^1}(\spanned(U^n),\spanned(U^*)) \leqslant C e^{-cn}.
		\end{align} 
	\end{theorem}
	\begin{proof}
		For any $u \in \text{span}(U^n)$ with $\|u\|_a = 1$, we have
		$$
		\text{dist}_{H_0^1}(u,\text{span}(U^*) ) = \|\mathcal{P}_{\bot}u\|_a.
		$$
		Suppose $u = \sum_{i = 1}^N \alpha_i u_i^n$, then 
		\begin{align*}
			\sum_{i = 1}^N \alpha_i^2 = \|u\|^2 \leqslant C \|u\|_a^2 = C.
		\end{align*}
		From Lemma \ref{lem: exponential convergence of numerical scheme}, we know that there exist constants $C$ and $c$ such that 
		$$
		\|\P_{\bot}U^n\|_a \leqslant C e^{-cn}.
		$$
		Hence,
		\begin{align*}
			\|\P_{\bot} u\|_a  \leqslant  \sum_{i = 1}^N |\alpha_i|\|\P_{\bot}u_i^n\|_a  \leqslant  C \|\P_{\bot} U^n\|_a  \leqslant C e^{-cn}, 
		\end{align*}
		which implies 
		\begin{equation*}
			\begin{aligned}
				{\delta}_{H_0^1}(\text{span}(U^n),\text{span}(U^*)) &= \sup_{u \in \text{span}(U^n), \,\|u\|_a = 1}  \|\mathcal{P}_{\bot}u\|_a \leqslant C e^{-cn}.
			\end{aligned}
		\end{equation*}
		This completes the proof.
	\end{proof}
	
	Recall the distance between $[U]$ and $[V]$ in $L^2$ sense
	\begin{align}
		\| [U] - [V]\| = \min_{Q \in \mathcal{O}^{N\times N}}\|U - VQ\|.
	\end{align}
	With the notion of \emph{principal angle} \cite{edelman1998geometry}, we have 
	\begin{align}
		\| [U] - [V]\|^2 = \sum_{j=1}^N 4\sin^2\frac{\theta_j}{2},
	\end{align}
	where $0 \leqslant \theta_1\leqslant  \cdots \leqslant  \theta_N \leqslant  \frac{\pi}{2} $ are the principal angles. Using these angles, the corresponding \emph{space distance} can be formulated as 
	\begin{align}
		{\delta}_{L^2}(\text{span}(U),\text{span}(V)) = \max_{k} \sin\theta_k = \sin\theta_N.
	\end{align}
	Following the proof of Theorem \ref{thm: exponential convergence of space distance}, we obtain 
	\begin{align*}
		{\delta}_{L^2}(\text{span}(U^n),\text{span}(U^*)) \leqslant C \|\P_{\bot}U^n\| \leqslant C \|\P_{\bot}U^n\|_a \leqslant C e^{-cn}.
	\end{align*}
	We immediately arrive at the following result, with the proof provided in Appendix \ref{proof of lemma: [Un] tends to [U_real]}.
	\begin{lemma}\label{lemma: [Un] tends to [U_real]}
		Under the assumption of Theorem \ref{thm: exponential convergence of space distance}, there holds 
		\begin{align}
			\| [U^n] - [U^*]\|_a \leqslant C e^{-cn}.
		\end{align}
	\end{lemma}

	Finally, we obtain the exponential energy convergence and \emph{orbital-to-orbital} convergence, consistent with the theoretical results of model \eqref{equ: evolution problem}. The \emph{orbital-wise} convergence is a further reinforcement of the exponential convergence of the subspaces presented in Theorem \ref{thm: exponential convergence of space distance}. This property ensures that the individual characteristics and properties of each orbital are preserved throughout the convergence process, rather than being lost in a collective subspace convergence.
	
	\begin{theorem}\label{thm: exponential convergence of Un by orbtial}
		Under the assumption of Theorem \ref{thm: exponential convergence of space distance}, there exist $C >0$, $c >0$ and $Q^*\in \mathcal{O}^N$ such that 
		\begin{equation*}
			\begin{aligned}
				E(U^n) - E(U^*) \leqslant C e^{-2cn} \quad \text{and} \quad	\|U^n - U^* Q^*\|_a \leqslant Ce^{-cn}.
			\end{aligned}
		\end{equation*}
	\end{theorem}
	\begin{proof}
		We see from Lemma \ref{lemma: [Un] tends to [U_real]} that, for any iteration step $n$, there exists a $Q^n\in \mathcal{O}^N$ such that $\|U^n - \bar{U}Q^n\|_a \leqslant Ce^{-cn}$ and 
		\begin{align}
			\L_{\bar{U}Q^n}\bar{U}Q^n = 0.
		\end{align}
		Then by the local Lipschitz continuity of $\L$, we conclude that
		\begin{align*}
			\|\L_{U^n}U^n\|_a  = \|\L_{U^n}U^n -	\L_{\bar{U}Q^n}\bar{U}Q^n \|_a 
			\leqslant C \|U^n - \bar{U}Q^n\|_a \leqslant Ce^{-cn}.
		\end{align*}
		Hence,
		\begin{align*}
			E(U^n) - E(U^{n+1}) \leqslant C(\tau_n + \tau_n^2) \|\L_{U^n}U^n\|_a^2 \leqslant C \tau_n e^{-2cn}.
		\end{align*}
		Since $E(U^n) \to E(U^*)$, we obtain
		\begin{align*}
			E(U^n) - E(U^*) & = \sum_{k=n}^{\infty} \left( E(U^k) - E(U^{k+1}) \right) \\
			& \leqslant C \sum_{k=n}^{\infty} e^{-2ck} = C \frac{\tau_{\text{max}} e^{-2cn}}{1 - e^{-2c}} = C e^{-2cn}.
		\end{align*}
		
		The numerical scheme implies that 
		\begin{align}
			\|U^{n+1} - U^n\|_a \leqslant \tau_n\|\L_{U^n}U^{n+\frac12}\|_a \leqslant C \tau_n\|\L_{U^n}U^{n}\|_a \leqslant \tau_n Ce^{-cn}.
		\end{align}
		This implies the following series 
		\begin{align}
			U^0 + \sum_{n = 0}^{\infty} (U^{n+1} - U^n)
		\end{align}
		strongly converges in $[H_0^1(\Omega)]^N$, which is $U^n \to \bar{U}$ strongly in $[H_0^1(\Omega)]^N$ for some $\bar{U} \in [H_0^1(\Omega)]^N$. Furthermore, $\langle \bar{U}, \bar{U} \rangle = I_N$ and $E(\bar{U}) = E_{\text{GS}}$. We can conclude that there exists a $Q^*\in \mathcal{O}^N$ such that $\bar{U} = U^*Q^*$. Therefore,
		\begin{equation*}
			\begin{aligned}
				\| U^n - U^*Q^*\|_a &=	\|U^n - \bar{U}\|_a =\left\| \sum_{k = n}^{\infty} (U^k - U^{k+1})\right\|_a	\leqslant \sum_{k=n}^{\infty}\|U^k - U^{k+1}\|_a \\
				&\leqslant C \tau_n \sum_{k=n}^{\infty} e^{-ck} \leqslant C \frac{ \tau_{\text{max}} e^{-cn}}{1 - e^{-c}} = C e^{-cn}.
			\end{aligned}
		\end{equation*}
		The proof is now complete.
	\end{proof}
	
	\section{Numerical experiments}\label{section:Numerical experiments}
	
	In this section, we evaluate the performance and effectiveness of the proposed algorithm using two typical eigenvalue problems: the harmonic oscillator and the three-dimensional Schr\"odinger equation for the hydrogen atom. All numerical experiments are carried out on the LSSC-IV platform at the Academy of Mathematics and Systems Science, Chinese Academy of Sciences.
	
	In practice, the model (\ref{equ: evolution problem}) can be discretized using various methods, such as the plane wave method, the finite difference method, or the finite element method. In this paper, we employ the finite element method for the spatial discretization, and all the results presented below are obtained using quadratic finite elements.
	
	Consider an $N_g$-dimensional space $V_{N_g} \subset H_0^1\left(\Omega\right)$ spanned by $\phi_1, \phi_2, \ldots, \phi_{N_g}$, and let $\Phi=\left(\phi_1, \phi_2, \ldots, \phi_{N_g}\right)$. For any $U \in\left(V_{N_g}\right)^N$, there exists $C \in \mathbb{R}^{N_g \times N}$ such that
	$$
	U=\Phi C=\left(\sum_{j=1}^{N_g} c_{j 1} \phi_j, \sum_{j=1}^{N_g} c_{j 2} \phi_j, \ldots, \sum_{j=1}^{N_g} c_{j N} \phi_j\right).
	$$
	In all the numerical experiments, the reference solution $(U^*,\Lambda^*)$ is obtained by using the existing eigenvalue solver to solve the eigenvalue problem (\ref{equ: model problem 1}) in the $N_g$-dimensional finite element space $V_{N_g}$. Unless stated otherwise, the iterations are terminated when the relative energy error 
    $$
	\text{err}_E^n=\frac{E(U^n)-E(U^*)}{E(U^*)}
	$$ 
    drops below $10^{-10}$. Time discretization adopts a uniform step size $\tau$, and the initial value $U^0$ is selected as random data with mutually orthogonal columns.
	
	For the clarity of presenting the subsequent numerical results, we denote the relative error of the approximate solutions as
	$$
	\text{err}_U^n=\frac{\|U^{n}-U_{\text{end}}\|}{\|U_{\text{end}}\|},
	$$
	with $U_{\text{end}}$ being the solution obtained from the final iteration (i.e., when the stopping criterion is met). We compute $\text{err}_U^n$ to demonstrate whether the convergence of eigenvectors is \emph{orbital-wise}. 
     Additionally, we denote the relative error of the approximate eigenvalues as
	$$
	\text{err}_i =\frac{\left|\lambda_i-\lambda_i^*\right|}{\left|\lambda_i^*\right|}, \quad i = 1,2, \cdots, N,
	$$
	where $\lambda_i$ are the eigenvalues of the matrix $\langle \mathcal{G}U_{\text{end}}, U_{\text{end}} \rangle^{-1}$, and $\lambda_i^*$ are the diagonal elements of $\Lambda^*$.

	\subsection{Test problem~I: Two–dimensional harmonic oscillator equation}
	
	We consider the following 2D harmonic oscillator equation \cite{ReedSimonIV}: Find $(u,\lambda) \in  H^1(\R^2)\times \mathbb{R}$ such that 
	\begin{equation}\label{eq:2D harmonic oscillator equation}
		-\frac{1}{2} \Delta u+\frac{1}{2}|x|^2 u = \lambda u, \qquad \int_{\R^2} u^2 = 1.
	\end{equation}
	where $|x| = \sqrt{x_1^2+x_2^2}$. The eigenvalues of (\ref{eq:2D harmonic oscillator equation}) are $\lambda_{n_1, n_2}=\left(n_1+\frac{1}{2}\right)+\left(n_2+\frac{1}{2}\right),\quad n_1, n_2=0,1, \cdots,$ and the corresponding eigenfunctions are
	$$
	u_{n_1, n_2}\left(x\right)=\mathcal{H}_{n_1}\left(x_1\right) e^{-x_1^2 / 2} \mathcal{H}_{n_2}\left(x_2\right) e^{-x_2^2 / 2}, \quad n_1, n_2=0,1, \cdots,
	$$
	where $\mathcal{H}_n$ denotes the $n$-th Hermite polynomial.
	
	Since the solution of (\ref{eq:2D harmonic oscillator equation}) decays exponentially, we may solve it over some bounded domain $\Omega$. In the computation, we solve the following eigenvalue problem: find $(u,\lambda) \in  H_0^1(\Omega) \times\mathbb{R} $ such that 
	\begin{equation}
		-\frac{1}{2} \Delta u+\frac{1}{2}|x|^2 u =\lambda u, \qquad \int_{\Omega} u^2  = 1.
	\end{equation}
	where $\Omega=(-5.5,5.5)^2$. We calculate the approximation of the first $N$ smallest eigenvalues with $N=15$ and their corresponding eigenfunctions. That is, we implement the model on a fixed uniform finite element mesh with degrees of freedom $N_g = 39601$, and adopt a fixed time step $\tau = 0.05$ for temporal discretization. Reference solutions are computed using the \emph{eigs} solver from \emph{Arpack.jl}.

	\begin{figure}[htbp]
		\centering
		\subfloat[Convergence curves of the energy]{
			\label{fig:Energy decay}%
			\includegraphics[width=0.49\linewidth]{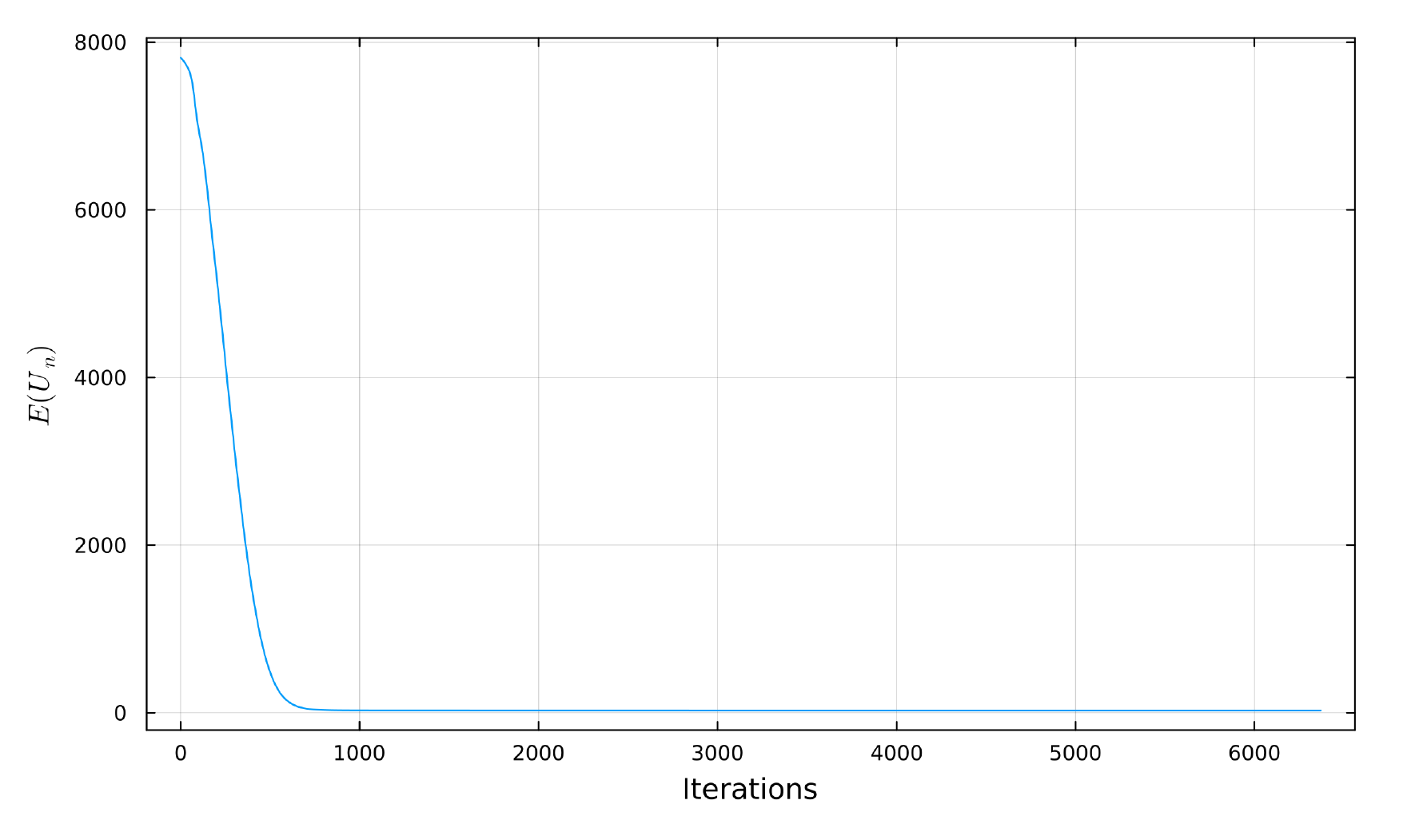}}
		\subfloat[Curves of orthogonality error]{%
			\label{fig:ortho err}%
			\includegraphics[width=0.49\linewidth]{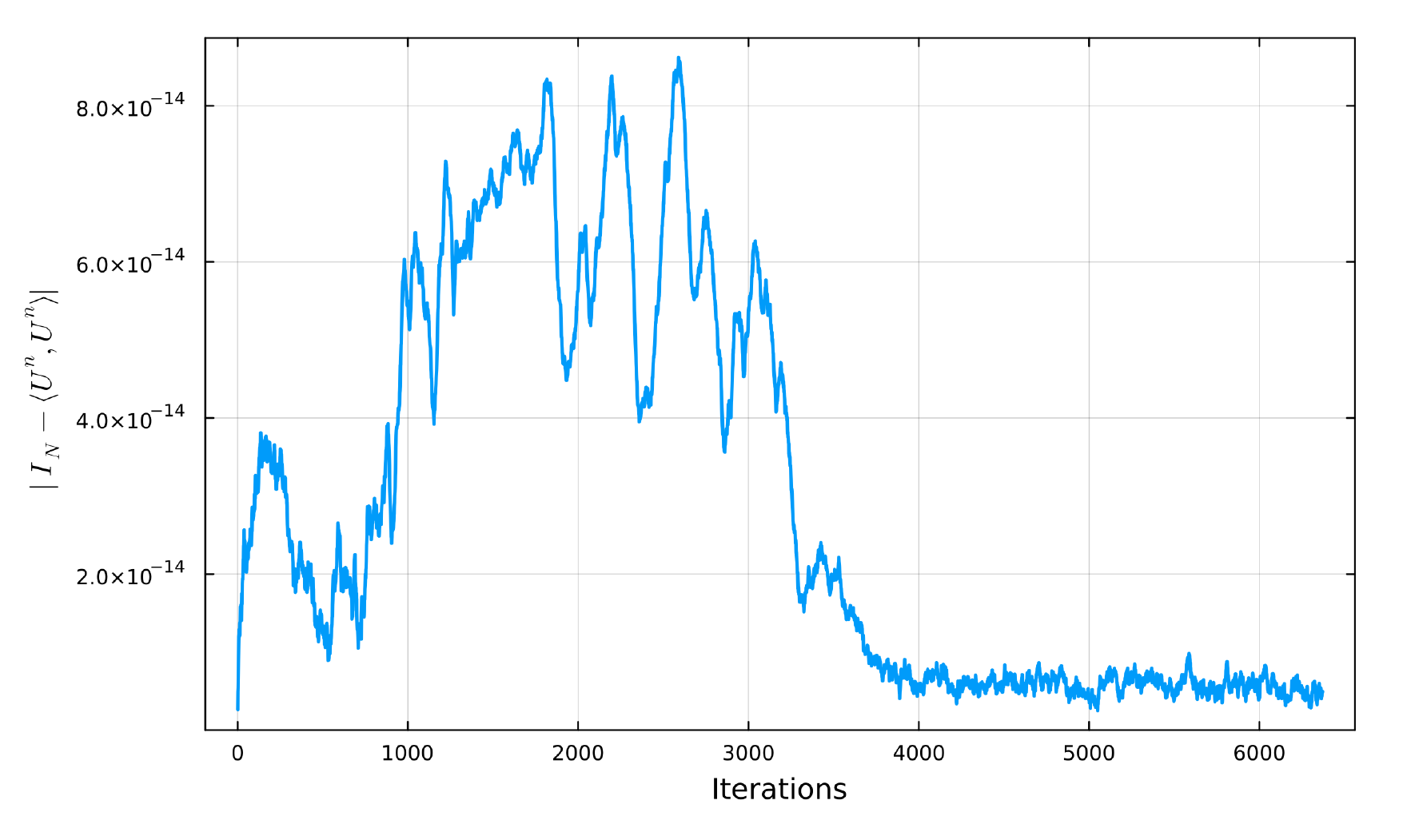}}
		
		\subfloat[Relative error curves of the energy]{%
			\label{fig:Energy exponential decay}%
			\includegraphics[width=0.49\linewidth]{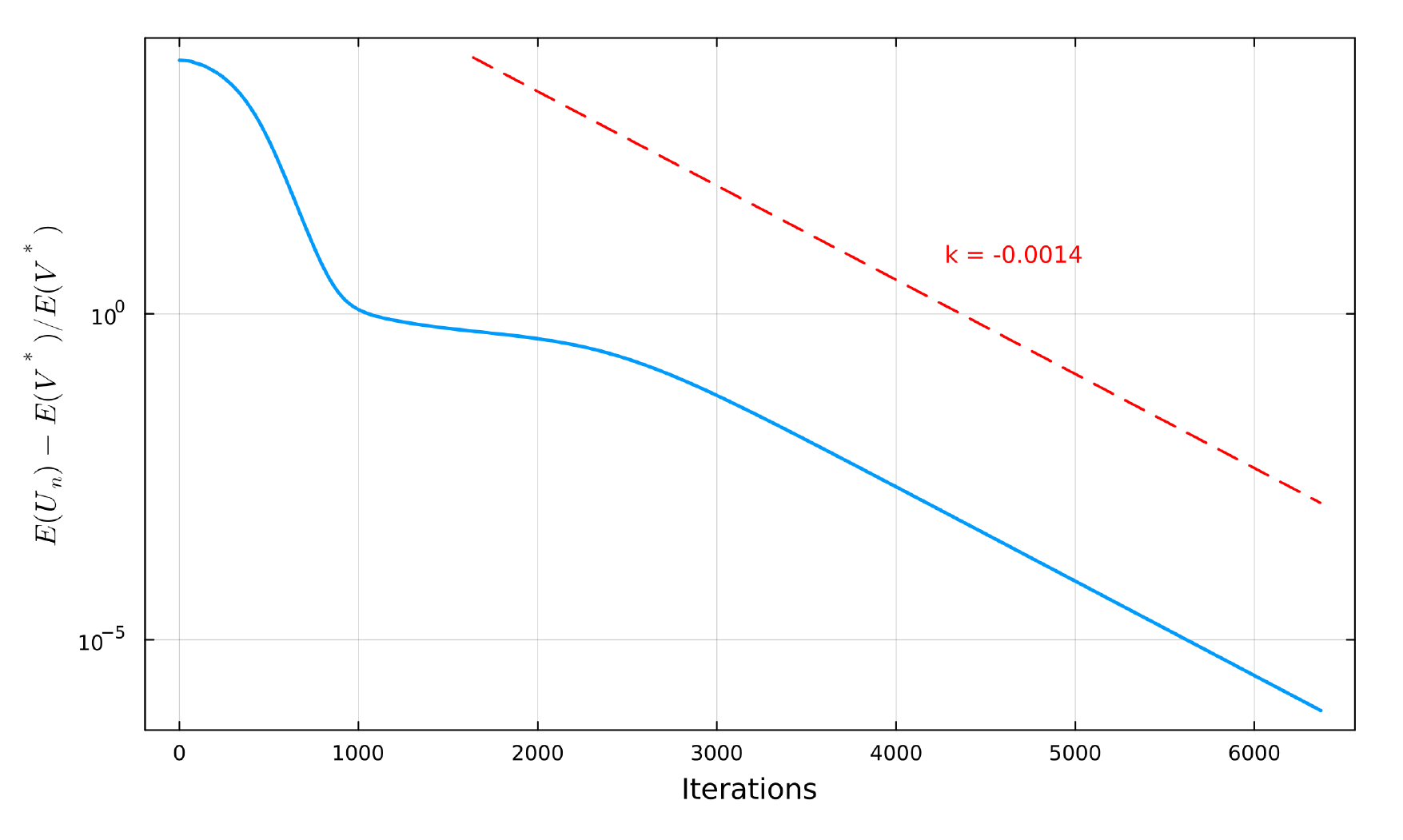}}
		\subfloat[Relative error curves of the eigenvectors]{%
			\label{fig:Eigenfunctions exponential convergence}%
			\includegraphics[width=0.49\linewidth]{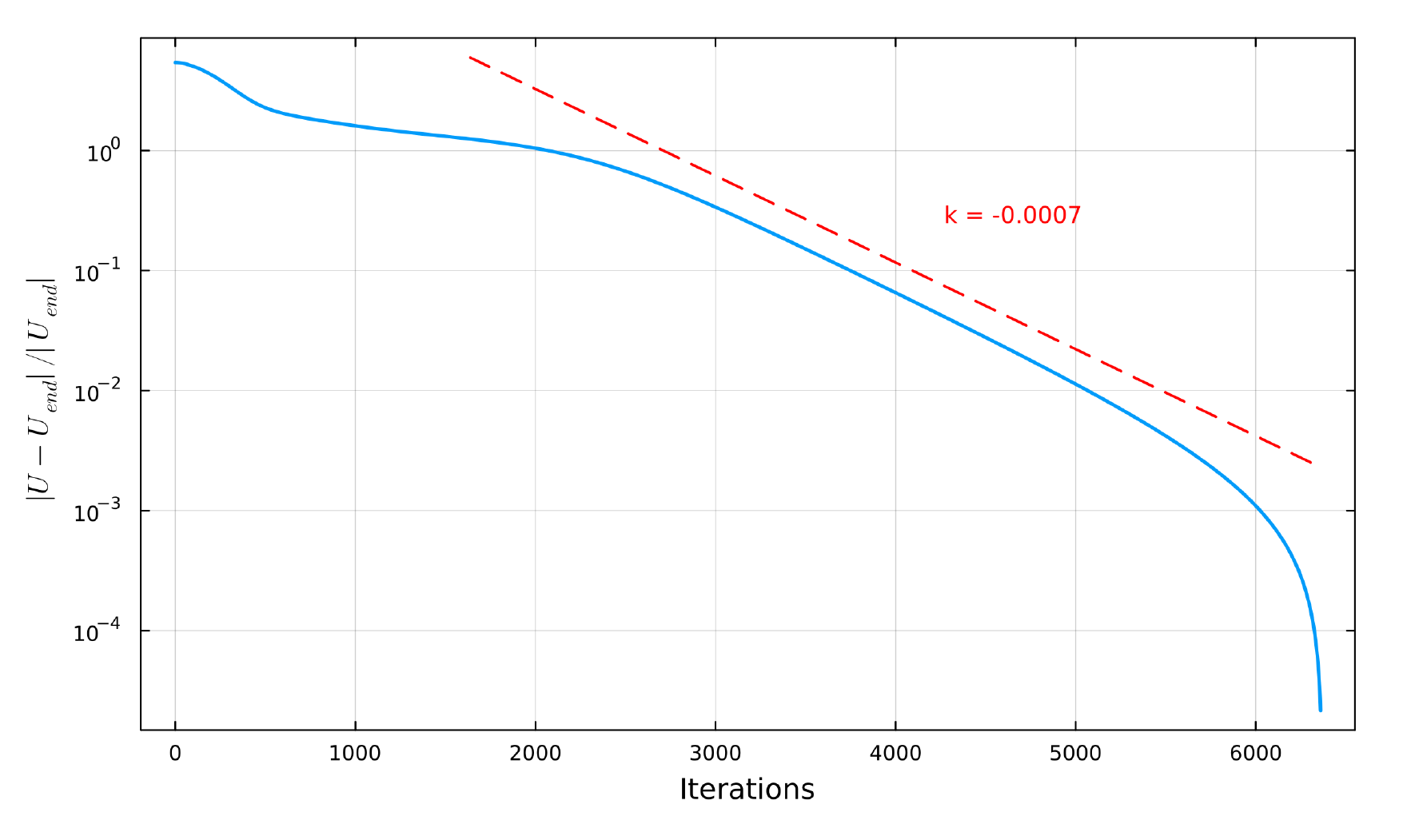}}
		\caption{Numerical results for test problem I}
	\end{figure}
	Figure~\ref{fig:Energy decay} displays the discrete energy $E(U^{n})$, which decreases monotonically with the iteration index $n$. The companion plot, Figure~\ref{fig:ortho err}, shows the orthogonality error $\bigl\|I_{N}-\langle U^{n},U^{n}\rangle\bigr\|$, demonstrating that the scheme preserves orthogonality.  Together, these observations corroborate Theorems~\ref{thm: Energy E(Un) decay} and~\ref{thm: orthogonality of Un}.
	
	The semi-log plots of the relative error of the energy $\text{err}_E^n$ (Figure~\ref{fig:Energy exponential decay}) and of the relative error of the solution $\text{err}_U^n$ (Figure~\ref{fig:Eigenfunctions exponential convergence}) exhibit a linear trend after a short transient period. The slope~$k$, extracted from the last $50\%$ of the samples, matches the asymptotic convergence rate predicted by Theorems~\ref{thm: exponential convergence of Un by orbtial}, thereby confirming the exponential convergence of both the discrete energy and the eigenvectors.
	At the same time, we can also find that the convergence of the eigenvectors is \emph{orbital-to-orbital}, and the exponential convergence rate of the energy is twice that of the eigenvectors.
	
	\begin{table}[htbp]
      \caption{Relative eigenvalue errors from different time step}
		\label{table:Mesh-independent time step}
		\tiny
		\centering
		\renewcommand{\arraystretch}{1.6}
		\begin{tabular}{|>{\centering\arraybackslash}m{0.5cm}|>{\centering\arraybackslash}m{1.77cm}|>{\centering\arraybackslash}m{1.65cm}|>{\centering\arraybackslash}m{1.65cm}|>{\centering\arraybackslash}m{1.53cm}|>{\centering\arraybackslash}m{1.53cm}|>{\centering\arraybackslash}m{1.53cm}|}
			\hline
			\multirow{2}{*}{\footnotesize $\text{err}_{i}$}&  $\bm{\tau = 0.01}$ & $\bm{\tau = 0.05}$ & $\bm{\tau = 0.1}$ & $\bm{\tau = 0.5}$ & $\bm{\tau = 1.0}$ & $\bm{\tau = 1.5}$ \\
			&  after 32425 steps &  after 6370 steps &  after 2866 steps &  after 681 steps &  after 301 steps &  after 209 steps \\
			\hline
			1  & 4.816e-13 & 1.776e-14 & 1.412e-13 & 6.051e-13 & 4.710e-13 & 7.194e-14 \\
			\hline
			2  & 4.128e-13 & 1.676e-13 & 1.750e-13 & 4.823e-13 & 2.343e-13 & 6.306e-14 \\
			\hline
			3  & 2.274e-13 & 2.764e-13 & 1.315e-13 & 4.106e-13 & 9.459e-14 & 1.517e-13 \\
			\hline
			4  & 1.048e-13 & 4.986e-13 & 2.747e-13 & 1.369e-13 & 1.085e-13 & 1.073e-13 \\
			\hline
			5  & 3.271e-14 & 1.172e-13 & 1.880e-14 & 5.921e-14 & 2.041e-13 & 4.367e-14 \\
			\hline
			6  & 8.275e-14 & 1.910e-14 & 1.782e-13 & 1.070e-13 & 1.211e-13 & 5.921e-16 \\
			\hline
			7  & 3.344e-13 & 1.878e-13 & 1.854e-13 & 2.838e-13 & 5.655e-13 & 3.142e-13 \\
			\hline
			8  & 1.004e-13 & 4.463e-14 & 1.432e-13 & 2.640e-13 & 2.098e-13 & 1.725e-13 \\
			\hline
			9  & 1.794e-13 & 2.720e-13 & 1.941e-13 & 5.185e-13 & 2.576e-13 & 3.078e-13 \\
			\hline
			10 & 1.401e-13 & 2.349e-13 & 1.554e-13 & 4.754e-13 & 6.928e-14 & 1.299e-13 \\
			\hline
			11 & 1.427e-08 & 1.454e-08 & 7.510e-08 & 1.280e-08 & 4.472e-09 & 9.254e-11 \\
			\hline
			12 & 3.331e-11 & 4.692e-09 & 1.676e-09 & 1.479e-10 & 4.612e-10 & 1.990e-11 \\
			\hline
			13 & 9.477e-07 & 2.238e-08 & 3.030e-08 & 6.185e-09 & 4.486e-09 & 2.803e-09 \\
			\hline
			14 & 5.834e-11 & 1.030e-09 & 3.247e-10 & 3.811e-12 & 2.057e-10 & 2.404e-10 \\
			\hline
			15 & 4.120e-07 & 2.537e-08 & 2.851e-08 & 6.790e-09 & 3.078e-09 & 5.074e-09 \\
			\hline
		\end{tabular}
	\end{table}
	Table~\ref{table:Mesh-independent time step} lists the relative errors of the computed eigenvalues $\text{err}_{i}, \ i = 1, 2,\cdots,N$, which are obtained on the fixed finite element mesh (fixed degree of freedom $N_g = 39601$) while the time step size $\tau$ is progressively increased. The accuracy remains essentially unchanged and no loss of stability is detected, thereby verifying that the admissible time step is \emph{mesh-independent}; in particular, the scheme is not subject to any CFL conditions. Moreover, a larger $\tau$ reduces the number of iterations and the wall-clock time required for convergence, yielding a significant overall speed-up of the algorithm.
	
	\subsection{Test problem~II: Three–dimensional Schr\"odinger equation for hydrogen atom}
	
	Consider the Schr\"odinger equation for hydrogen atoms \cite{greiner2011quantum}: Find $(u,\lambda) \in H^1(\R^3) \times \R$ such that 
	\begin{equation}\label{eq:3D hydrogen}
		\left(-\frac{1}{2} \Delta-\frac{1}{|x|}\right) u=\lambda u, \qquad \int_{\mathbb{R}^3}|u|^2  = 1.
	\end{equation}
	The eigenvalues of (\ref{eq:3D hydrogen}) are $\lambda_n=-\frac{1}{2 n^2}$ $(n=1,2, \cdots)$ and the multiplicity of $\lambda_n$ is $n^2$.
	
	Since the eigenvectors of \eqref{eq:3D hydrogen} decay exponentially, instead of (\ref{eq:3D hydrogen}), we may solve the following eigenvalue problem: Find $(u,\lambda) \in H_0^1(\Omega)\times \mathbb{R}$ such that 
	\begin{equation}\label{eq:3D bound hydrogen}
		\left(-\frac{1}{2} \Delta-\frac{1}{|x|}\right) u =\lambda u, \qquad \int_{\Omega} u^2 = 1,
	\end{equation}
	where $\Omega$ is some bounded domain in $\mathbb{R}^3$. In our computation, we choose $\Omega=(-20.0,20.0)^3$ and compute approximations of the first 2 smallest eigenvalues and their corresponding eigenvector space approximations. Since the multiplicity of the $n$-th smallest eigenvalue is $n^2$, for the discrete problem of (\ref{eq:3D bound hydrogen}), we calculate the first 5 smallest eigenvalues and their associated eigenvectors.  We adopt the adaptive finite element method \cite{dai2015convergence} to deal with the spatial discretization with degrees of freedom $N_g = 570662$, and use a fixed time step $\tau = 1.0$ for temporal discretization. The reference eigenvalues $\lambda^*_i$ and residual norms $r_i$ for $i = 1, 2, \cdots,N$ in Table \ref{table:reference eigenvalues} were computed using the \emph{eigsolve} solver from \emph{IterativeSolvers.jl}, where $r_i = \|\mathcal{H}u_i -\lambda^*_iu_i\|,\  i = 1, 2, \cdots,N$.
	
	\begin{table}[H]
		\centering
		\small
		\begin{minipage}[t]{0.48\linewidth}
			\centering
			\caption{Reference eigenvalues $\lambda_i^*$ and its residual norm $r_i$}
			\label{table:reference eigenvalues}
			\begin{tabular}{ccc}
				\hline
				$i$ & $\lambda_i^*$ & $r_i$ \\
				\hline
				1 & $-0.4999583481345601$ & $2.605\times10^{-9}$ \\
				2 & $-0.1249998780617823$ & $3.873\times10^{-7}$ \\
				3 & $-0.1249998492802271$ & $3.828\times10^{-7}$ \\
				4 & $-0.1249992663791233$ & $1.244\times10^{-4}$ \\
				5 & $-0.1249961959501441$ & $2.793\times10^{-6}$ \\
				\hline
				&\multicolumn{2}{l}{$E(U^*) = -0.49997676890291845$} \\
				\hline
			\end{tabular}
		\end{minipage}\hfill
		\begin{minipage}[t]{0.52\linewidth}
			\centering
			\caption{Approximate eigenvalues $\lambda_i$ and their relative errors $\text{err}_{i}$}
			\label{tab:compute 5 eigenvalues of hydrogen}
			\begin{tabular}{ccc}
				\hline
				$i$ & $\lambda_i$ & $\text{err}_{i}$ \\
				\hline
				1 & $-0.49995834814124984$ & $1.337\times10^{-11}$ \\
				2 & $-0.12499988213671587$ & $3.244\times10^{-6}$ \\
				3 & $-0.12499986154701304$ & $9.334\times10^{-5}$ \\
				4 & $-0.1249998518636215$ & $6.827\times10^{-7}$ \\
				5 & $-0.12499619600148248$ & $4.114\times10^{-9}$ \\
				\hline
				&\multicolumn{2}{l}{$E(U_\text{end})= -0.4999700698450414$} \\
				\hline
			\end{tabular}
		\end{minipage}
	\end{table}

	Similarly, for the Schr\"odinger equation for hydrogen atoms, the discrete energy $E(U^{n})$ decreases monotonically with the iterations (as shown in Figure~\ref{fig:Energy decay of hydrogen}), and the scheme preserves orthogonality (as shown in Figure~\ref{fig:ortho err of hydrogen}). These observations once again corroborate Theorems~\ref{thm: Energy E(Un) decay} and~\ref{thm: orthogonality of Un}.
	
	We also observe the exponential convergence of both the discrete energy (as shown in Figure~\ref{fig:Energy exponential decay of hydrogen}) and the eigenvectors (as shown in Figure~\ref{fig:Eigenfunctions exponential convergence of hydrogen}), where the convergence of the eigenvectors is \emph{orbital-to-orbital}.
	 \begin{figure}[H]
		\centering
		\subfloat[Convergence curves of the energy]{\label{fig:Energy decay of hydrogen}\includegraphics[width=0.49\linewidth]{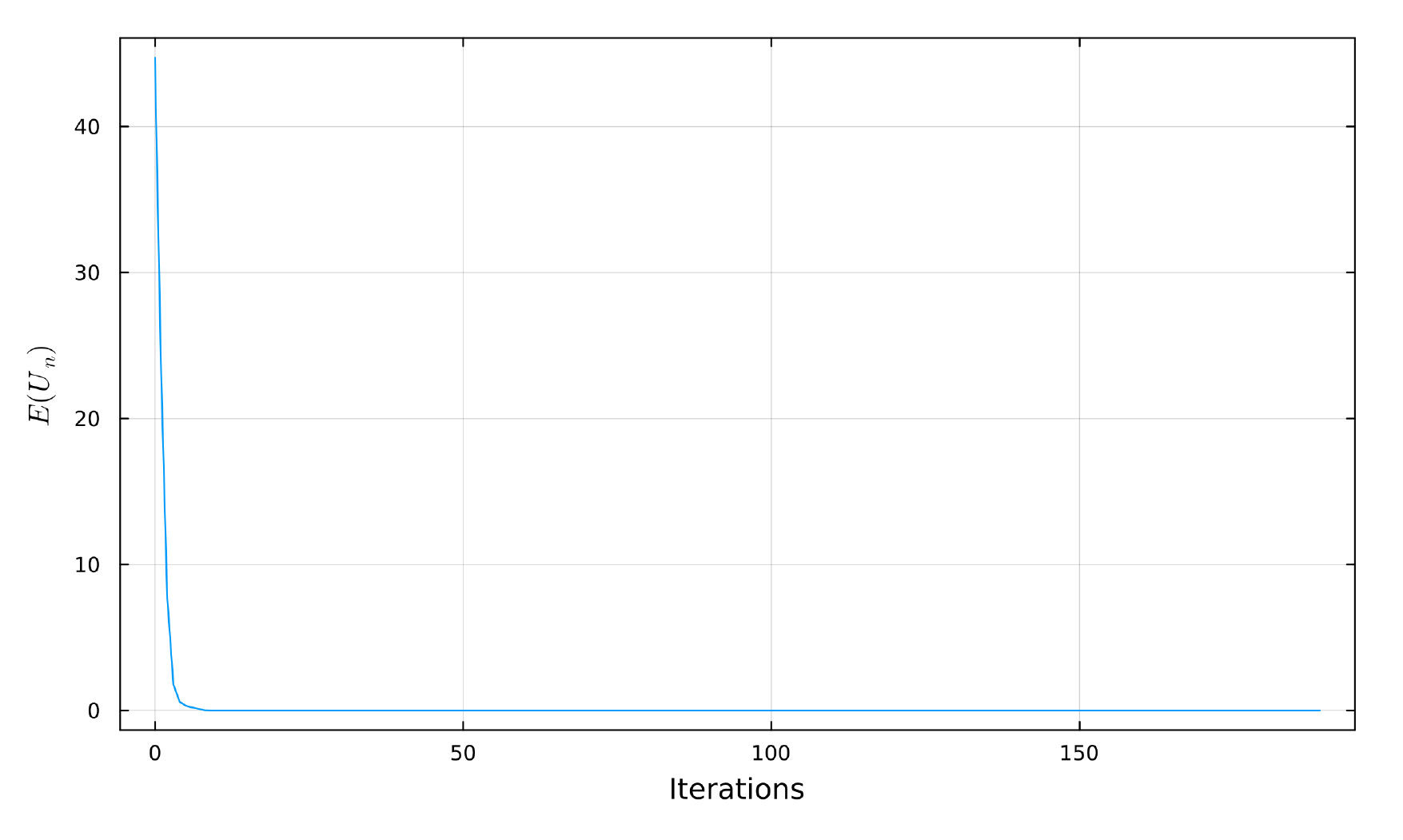}}
		\subfloat[Curves of orthogonality error]{\label{fig:ortho err of hydrogen}\includegraphics[width=0.49\linewidth]{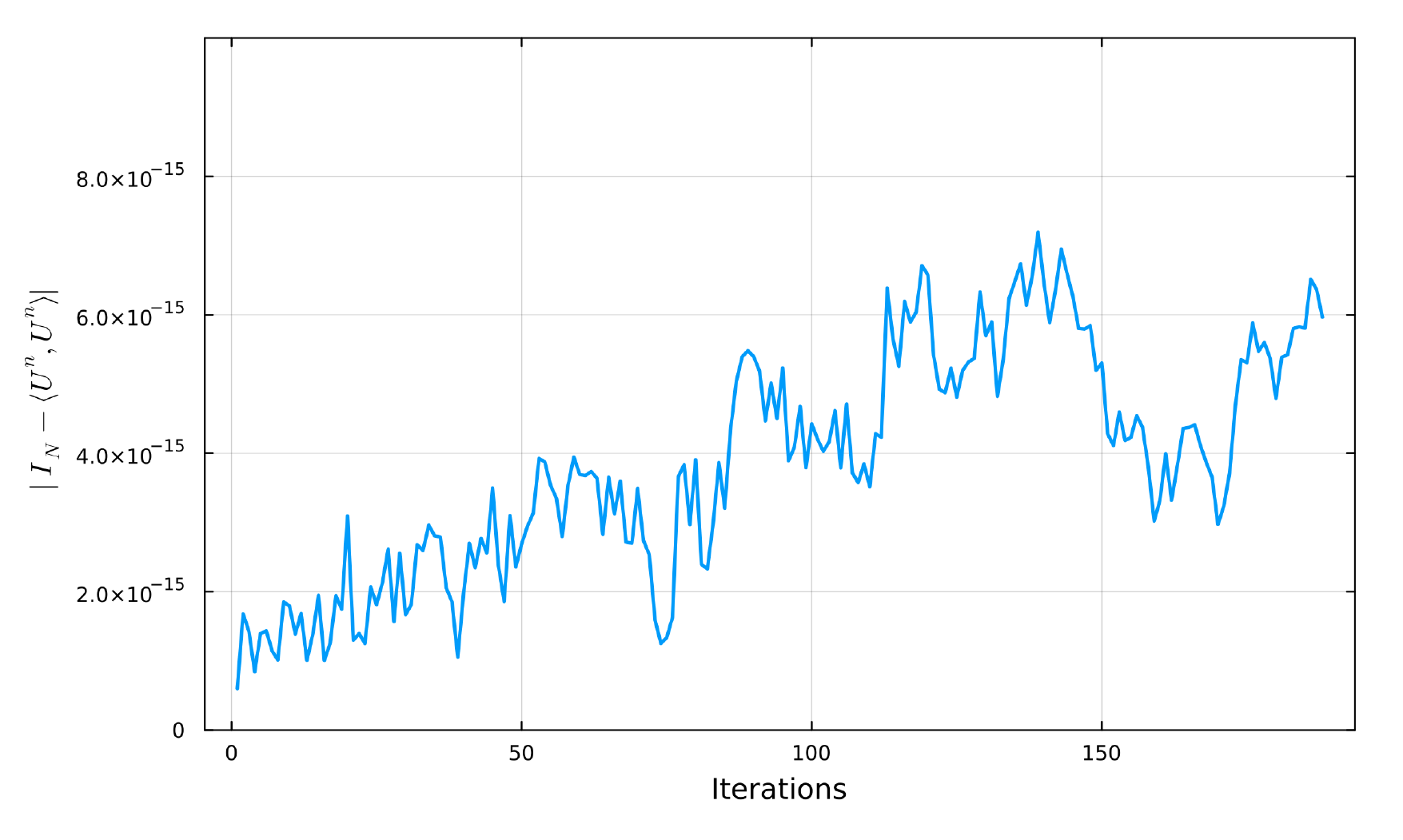}}
		
		\subfloat[Relative error curves of the energy]{\label{fig:Energy exponential decay of hydrogen}\includegraphics[width=0.49\linewidth]{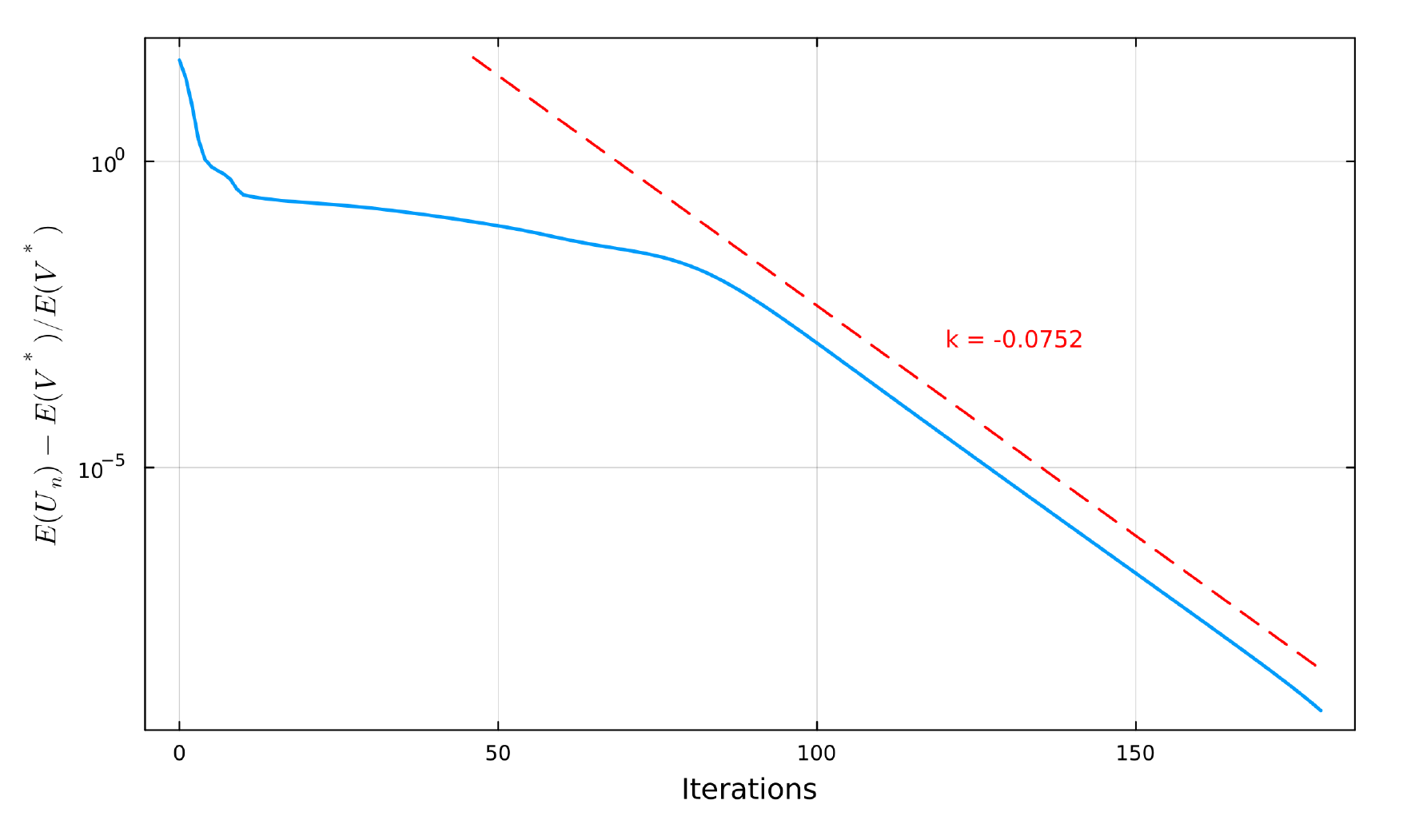}}
		\subfloat[Relative error curves of the eigenvectors]{\label{fig:Eigenfunctions exponential convergence of hydrogen}\includegraphics[width=0.49\linewidth]{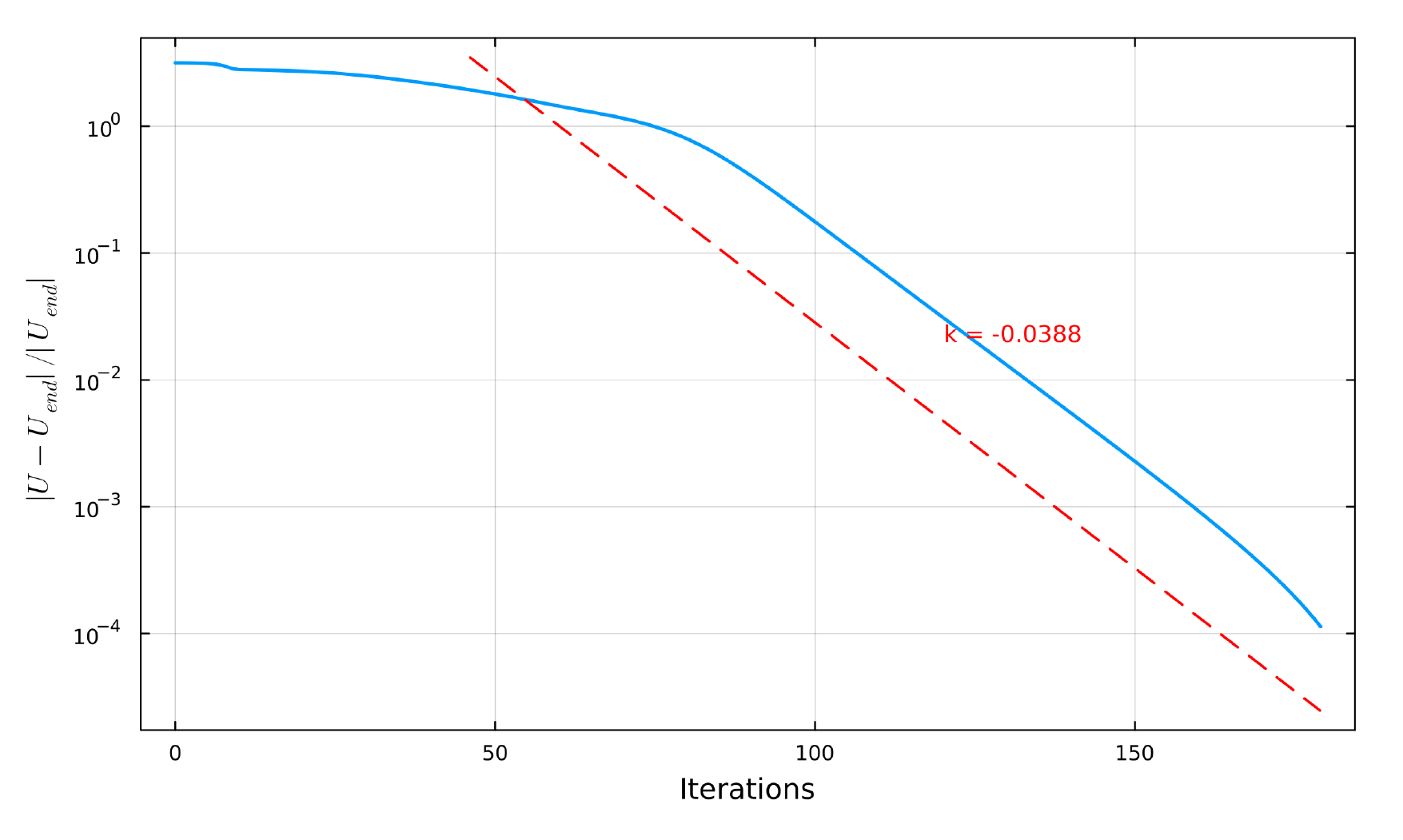}}
		\caption{Numerical results for test problem II}
	\end{figure}
	
	Table~\ref{tab:compute 5 eigenvalues of hydrogen} lists the relative errors of the first five eigenvalues, and Figure~\ref{fig:eigenvector of hydrogen} shows the 2D slice heatmaps of the numerical solution. Collectively, these findings qualitatively confirm the efficiency of our approach proposed in this paper, highlighting the robustness of the model when addressing large-scale three-dimensional problems.
	
	\begin{figure}[H]
		\centering
		\subfloat[eigenvector corresponding to $\lambda_1$: $u_1$]{\label{fig:H5:1}\includegraphics[width = 0.6\linewidth]{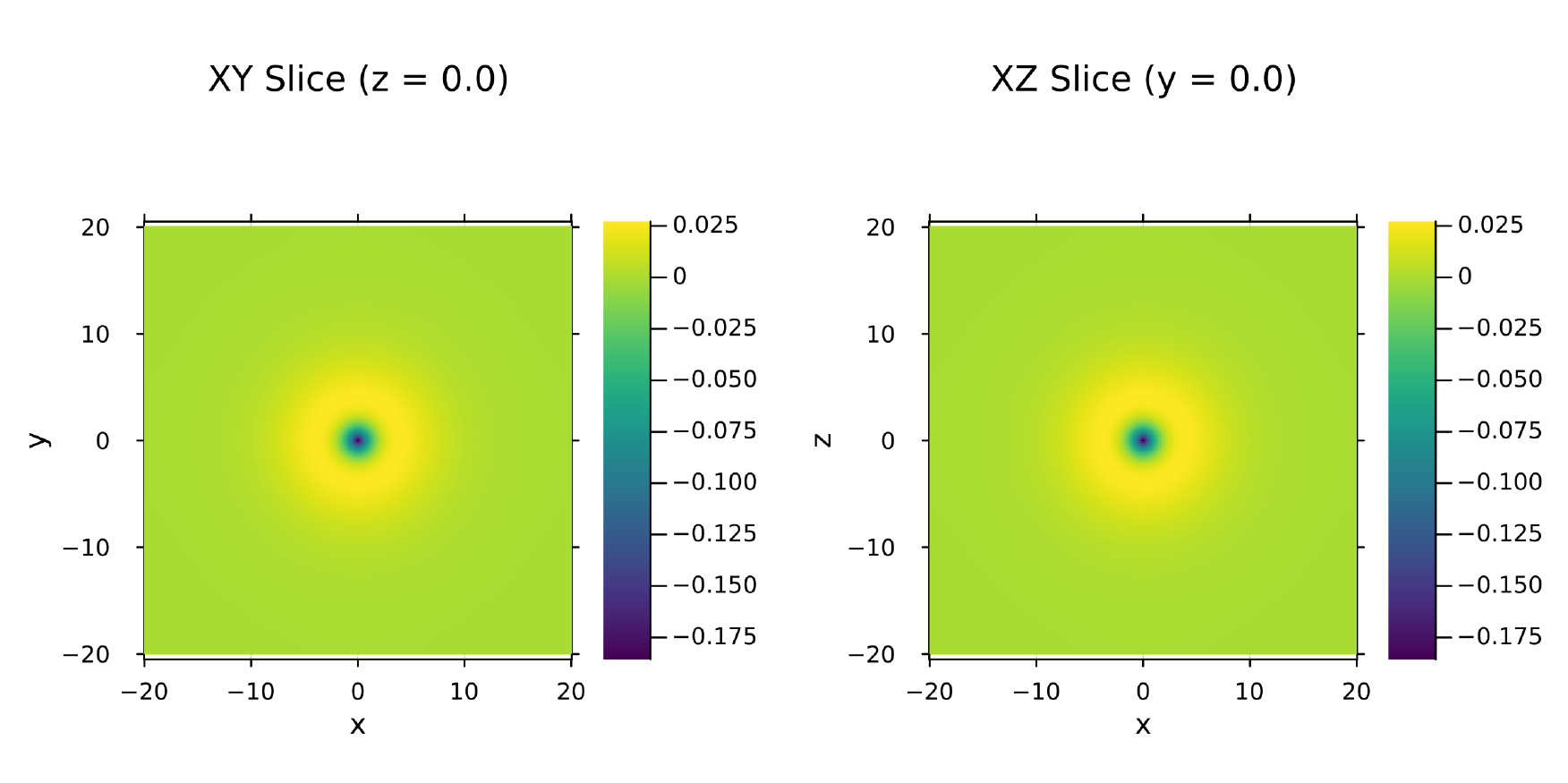}}
		
		\subfloat[eigenvector corresponding to $\lambda_2$: $u_2$]{\label{fig:H5:2}\includegraphics[width = 0.49\linewidth]{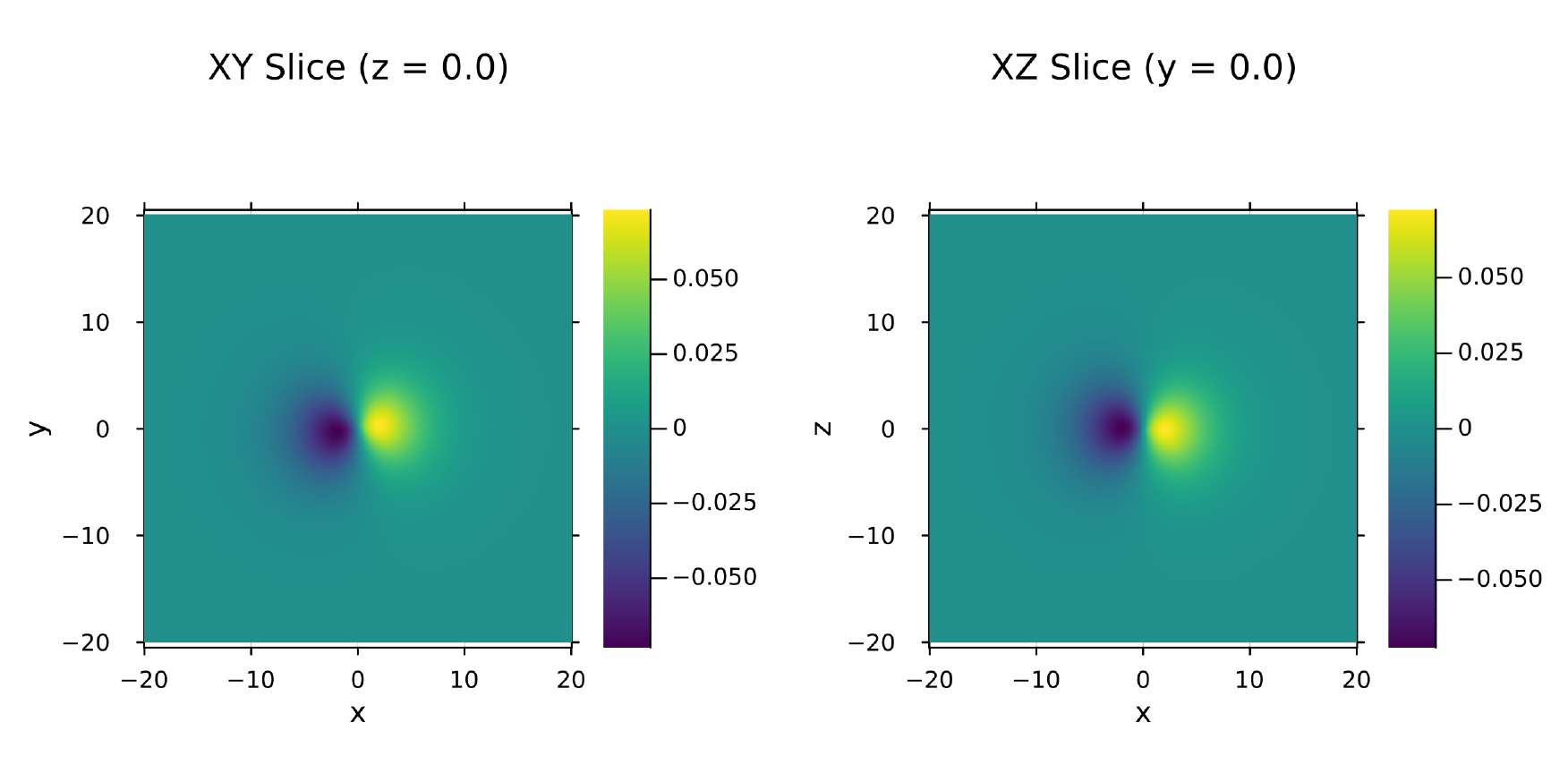}}
		\subfloat[eigenvector corresponding to $\lambda_2$: $u_3$]{\label{fig:H5:3}\includegraphics[width = 0.49\linewidth]{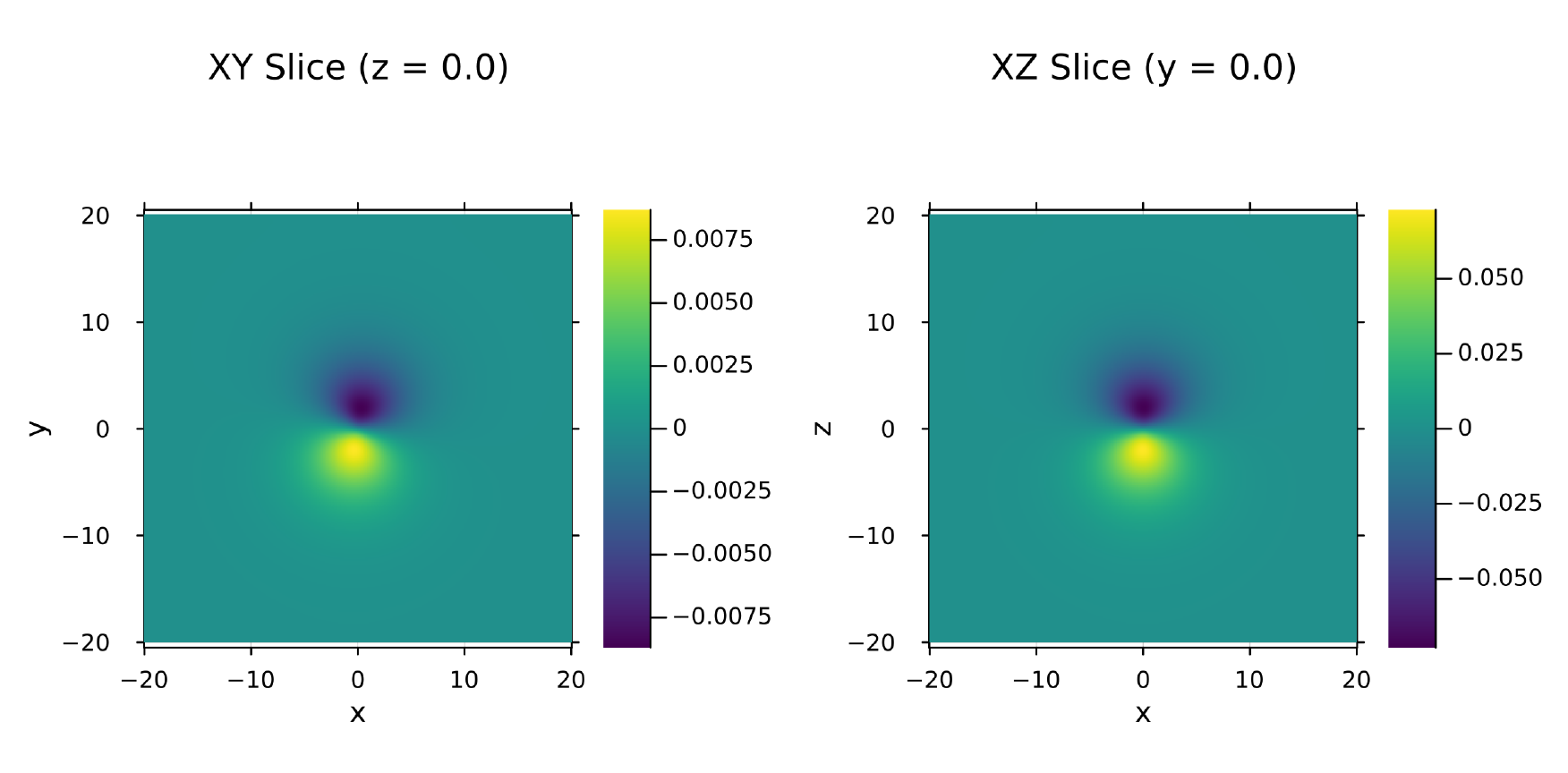}}
		
		\subfloat[eigenvector corresponding to $\lambda_2$: $u_4$]{\label{fig:H5:4}\includegraphics[width = 0.49\linewidth]{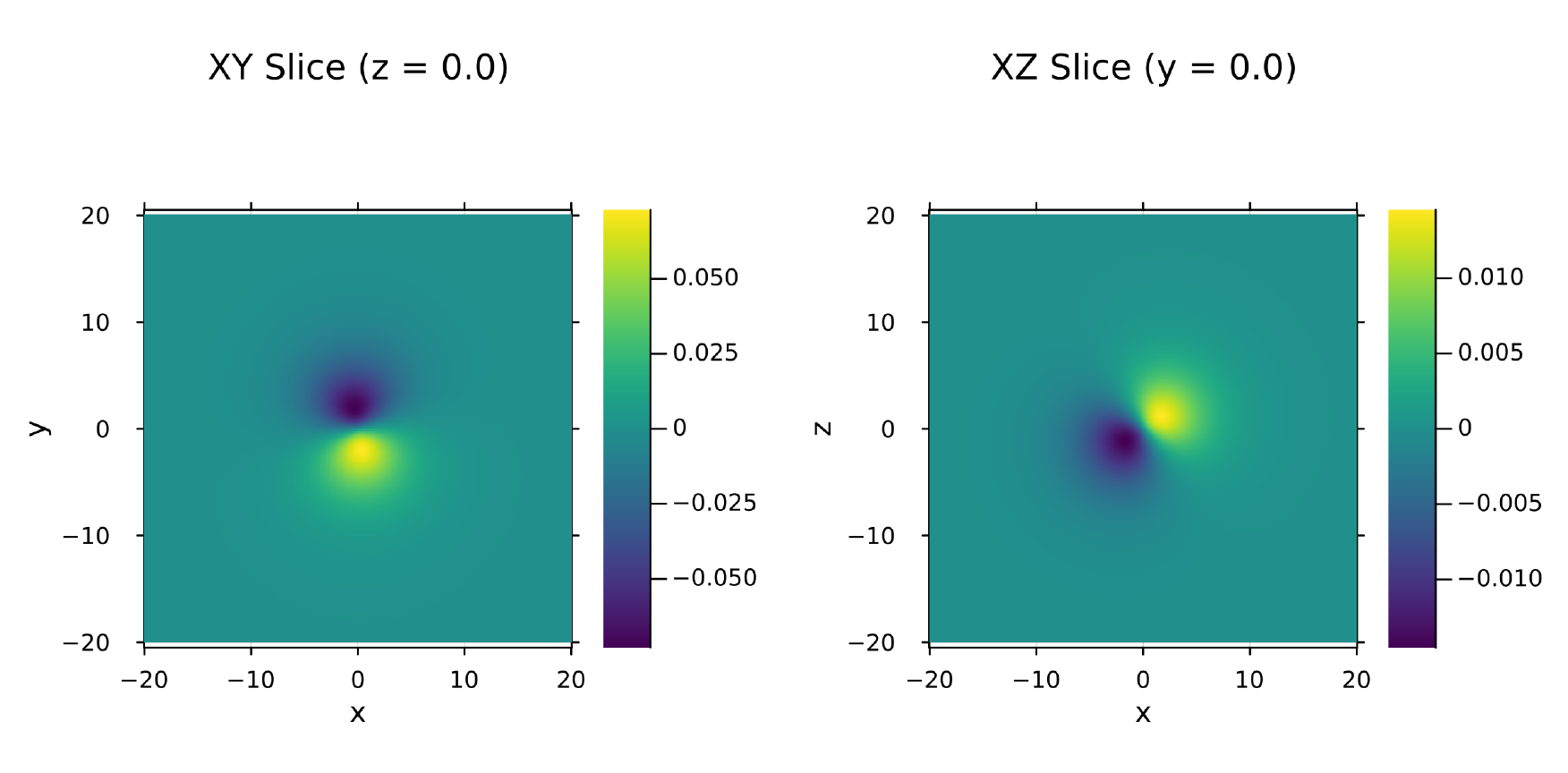}}
		\subfloat[eigenvector corresponding to $\lambda_2$: $u_5$]{\label{fig:H5:5}\includegraphics[width = 0.49\linewidth]{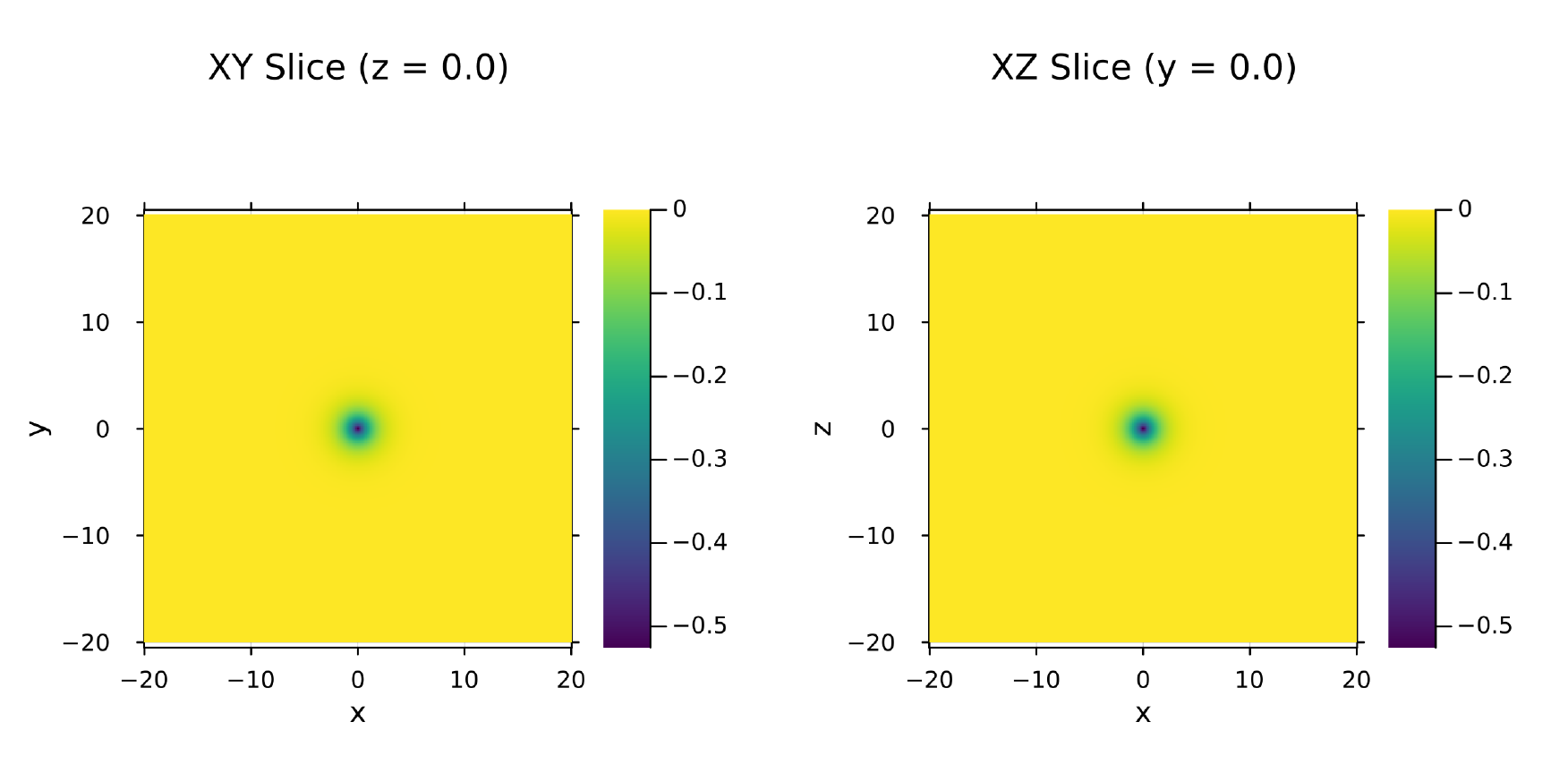}}
		
		\caption{Two-dimensional slice heatmap of five eigenvectors: (a) correspond to the first single eigenvalue $\lambda_1$; (b)--(e) correspond to the second eigenvalue $\lambda_2$ with quadruple degenerate.}
		\label{fig:eigenvector of hydrogen}
	\end{figure}
	
	The numerical experiments confirm several key aspects of the theoretical developments established in Section~\ref{section:Time discretization}. The results demonstrate strict energy dissipation and maintenance of orthogonality. Moreover, the numerical experiments exhibit exponential convergence of both the energy and eigenvectors, highlighting the efficiency of the algorithm. Notably, the time-step constraint is independent of the spatial discretization, allowing for flexible and efficient time-stepping schemes. Furthermore, the method possesses high accuracy and efficiency for realistic three-dimensional quantum models, underscoring its applicability to complex problems.
	
	\section{Conclusion}\label{section:Conclusion}
	To address the computational complexity and parallel scalability limitations caused by orthogonalization operations, particularly for eigenvalue problems requiring many eigenpairs, we have proposed an intrinsic orthogonality-preserving model structured as an evolution equation. Based on this model, we developed a numerical method that automatically preserves orthogonality and is energy-dissipative throughout the iteration process. We rigorously proved the convergence of the proposed model and numerical method. Numerical experiments validate the theoretical analyses and demonstrate the high efficiency of the proposed algorithm. 

The proposed method offers a promising approach for efficiently computing many eigenpairs for large scale eigenvalue problems with orthogonality restrictions. Our ongoing work focuses on developing more robust orthogonality-preserving schemes and on extending the approach to nonlinear settings; the latter demands sophisticated analysis and will be treated elsewhere.
	
	\appendix
	\section{Detailed proofs}\label{app:proofs}
	
	    \subsection{Proof of Lemma~\ref{lemma:B_n exists}}\label{proof of lemma:B_n exists}
		
		\begin{proof}
			Choose $\bm{\alpha}  \in \R^N$ with $|\bm{\alpha} | = 1$, and denote $u = \sum_{i = 1}^N \alpha_i u_i$. Then, 
			\begin{align*}
				\bm{\alpha} \big(\langle  \mathcal{G}U, \mathcal{G} U \rangle -   \langle  \mathcal{G}U,  U \rangle \langle \mathcal{G}U,  U \rangle\big)\bm{\alpha} ^{\top} 
				=  (\G u, \G u) - \sum_{i=1}^N (u_i, \G u)^2,
			\end{align*}
			which implies that
			\begin{align*}
				(\G u, \G u) - \sum_{i=1}^N (u_i, \G u)^2 \geqslant 0
			\end{align*}
			and completes the proof.
		\end{proof}
        
	\subsection{Proof of Lemma~\ref{lem: skew-symmetry}}\label{proof of lemma: skew-symmetry}
	\begin{proof}
		For any $V,W\in [H_0^1(\Omega)]^N$, the following equality holds:
		\begin{align*}
			\langle V, \mathcal{L}_U W \rangle + \langle \mathcal{L}_U V, W \rangle &  = \langle V, U \rangle \langle \G U,  W \rangle - \langle V, \G U \rangle\langle U, W \rangle \\
			& \quad + \langle V, \G U \rangle \langle U, W\rangle - \langle V, U \rangle \langle \G U , W \rangle \\
			& = 0.
		\end{align*}
		Furthermore, if $U \in \mathcal{M}^N$, then
		\begin{equation*}
			\begin{aligned}
				\langle U, \mathcal{L}_U U \rangle  &=  \langle U, U \rangle \langle \G U,  U \rangle - \langle U, \G U \rangle\langle U, U \rangle 
				\\& = \langle \G U,  U \rangle - \langle U, \G U \rangle = 0.
			\end{aligned}
		\end{equation*}
		Thus, the proof is complete.
	\end{proof}

	\subsection{Proof of Lemma~\ref{lem: bound property of L_U}}\label{proof of lem: bound property of L_U}
		\begin{proof}
			Note that the $i$-th component of $\mathcal{L}_UV$ is 
			\begin{align*}
				[\mathcal{L}_UV]_i = \sum_{j = 1}^N u_j (\G u_i, v_j) - \sum_{j = 1}^N \G u_j ( u_i, v_j).
			\end{align*}
			It follows that
			\begin{align*}
				\|[\mathcal{L}_UV]_i\|_a  \leqslant \sum_{j = 1}^N |(\G u_i, v_j)| \|u_j\|_a + \sum_{j = 1}^N |( u_i, v_j)| \|\G u_j\|_a  \leqslant C \|U\|_a^2 \sum_{j = 1}^N \|v_j\|.
			\end{align*}
			Consequently, the following inequality holds: 
			\begin{align*}
				\|[\mathcal{L}_UV]_i\|_a^2 \leqslant C N \|U\|_a^4 \|V\|^2.
			\end{align*}
			Summing over $i$ from $1$ to $N$ leads to 
			\begin{align*}
				\|\L_U V \|_a^2 \leqslant C \|U\|_a^4  N^2 \|V\|^2,
			\end{align*}
			which completes the proof.
		\end{proof}

		\subsection{Proof of Lemma~\ref{lem: Local Lipschitz property}}\label{proof of lem: Local Lipschitz property}
		\begin{proof}
			The $i$-th component of $\mathcal{L}_UU$ is given by
			\begin{align*}
				[\mathcal{L}_UU]_i = \sum_{j = 1}^N u_j (\G u_i, u_j) - \sum_{j = 1}^N \G u_j ( u_i, u_j).
			\end{align*}
			It follows that 
			\begin{align*}
				\| \mathcal{L}_U U - \mathcal{L}_V V \|_a^2 = \sum_{i = 1}^N \| [\mathcal{L}_UU]_i - [\mathcal{L}_VV]_i \|_a^2.
			\end{align*}
			The difference $[\mathcal{L}_UU]_i - [\mathcal{L}_VV]_i$ can be decomposed as follows:
			\begin{align*}
				[\mathcal{L}_UU]_i - [\mathcal{L}_VV]_i & = \sum_{j = 1}^N (u_j - v_j) (\G u_i, u_j) + \sum_{j = 1}^N v_j \big((\G u_i, u_j) - (\G v_i, v_j)\big) \\
				& \quad - \sum_{j = 1}^N ( \G u_j - \G v_j )(u_i,u_j) - \sum_{j=1}^N \G v_j \big( (u_i,u_j) - (v_i,v_j) \big).
			\end{align*}
			Observe that
			\begin{align*}
				|(\G u_i, u_j) - (\G v_i, v_j)| & \leqslant |(\G u_i - \G v_i, u_j)| + |(\G v_i, u_j - v_j)| \\
				& \leqslant \|\G u_i - \G v_i\|\|u_j\| + \|\G v_i\|\|u_j - v_j\| \\
				& \leqslant C_M \|u_i - v_i\| + C_M \|u_j - v_j\|,
			\end{align*}
			and similarly,
			\begin{align*}
				|( u_i, u_j) - ( v_i, v_j)| & \leqslant |( u_i -  v_i, u_j)| + |( v_i, u_j - v_j)| \\
				& \leqslant \| u_i -  v_i\|\|u_j\| + \| v_i\|\|u_j - v_j\| \\
				& \leqslant C_M \|u_i - v_i\| + C_M \|u_j - v_j\|.
			\end{align*}
			By combining these estimates, we obtain
			\begin{align*}
				\| [\mathcal{L}_UU]_i - [\mathcal{L}_VV]_i\|_a & \leqslant \sum_{j =1}^N C_M \|u_j - v_j\|_a + \sum_{j =1}^N C_M (\|u_i - v_i\| +  \|u_j - v_j\|) \\
				& \quad + \sum_{j =1}^N C_M \|u_j - v_j\|  + \sum_{j =1}^N C_M (\|u_i - v_i\| +  \|u_j - v_j\|) \\
				& \leqslant N C_M \|u_i - v_i\|_a + C_M \sum_{j =1}^N  \|u_j - v_j\|_a.
			\end{align*}
			Applying H\"older's inequality further leads to
			\begin{align*}
				\| [\mathcal{L}_UU]_i - [\mathcal{L}_VV]_i\|_a^2 \leqslant C_M N^2 \|u_i - v_i\|_a^2 + C_M N \sum_{j =1}^N  \|u_j - v_j\|_a^2.
			\end{align*}
			Summing $i$ from $1$ to $N$ yields
			\begin{align*}
				\| \mathcal{L}_U U - \mathcal{L}_V V \|_a^2 \leqslant C_M N^2 \|U - V\|_a^2,
			\end{align*}
			which completes the proof.
		\end{proof}

		\subsection{Proof of Lemma~\ref{lem: uniform bound}}\label{proof of lem: uniform bound}
		\begin{proof}
			Let $\bm{\alpha}  \in \R^N$ be an arbitrary vector with $|\bm{\alpha} | = 1$. By Proposition \ref{prop: orthogonality preserving}, we have for all $t \in [0,T)$, $\langle U(t), U(t) \rangle = I_N$, which implies $\bm{\alpha} \langle U(t), U(t) \rangle \bm{\alpha}^{\top} = 1$, or equivalently, $\|\sum_{i = 1}^N \alpha_i u_i(t)\| = 1$.
			Using this result, we can obtain an upper bound for $\bm{\alpha} \langle \G U(t), U(t) \rangle \bm{\alpha} ^{\top}$:
			\begin{align*}
				\bm{\alpha} \langle \G U(t), U(t) \rangle \bm{\alpha} ^{\top}  = \left\|\G \left( \sum_{i =1}^N \alpha_i u_i(t) \right)\right\|_a^2 \leqslant C \left\|\sum_{i = 1}^N \alpha_i u_i(t)\right\|^2 \leqslant C,
			\end{align*}
			where $C>0$ is independent of $t$.
			
			Similarly, a time-independent lower bound for $\bm{\alpha} \langle \G U(t), U(t) \rangle \bm{\alpha} ^{\top}$ can be derived:
			\begin{align*}
				1 = \left\|\sum_{i = 1}^N \alpha_i u_i(t)\right\|^2 & = \left(\sum_{i = 1}^N \alpha_i u_i(t), \sum_{i = 1}^N \alpha_i u_i(t)\right) = \left(\G (\sum_{i = 1}^N \alpha_i u_i(t)), \sum_{i = 1}^N \alpha_i u_i(t) \right)_a \\
				& \leqslant \left\| \G (\sum_{i = 1}^N \alpha_i u_i(t))\right\|_a \left\|\sum_{i = 1}^N \alpha_i u_i(t)\right\|_a \leqslant  \left\| \G (\sum_{i = 1}^N \alpha_i u_i(t))\right\|_a \|U(t)\|_a \\
				& \leqslant 2 E(U^0) \left\| \G (\sum_{i = 1}^N \alpha_i u_i(t))\right\|_a  = 2 E(U^0)\cdot \bm{\alpha} \langle \G U(t), U(t) \rangle \bm{\alpha}^{\top} .
			\end{align*}
			
			By combining the upper and lower bounds for $\bm{\alpha} \langle \G U(t), U(t) \rangle \bm{\alpha} ^{\top}$, it can be concluded that there exist constants $C_1$ and $C_2$ such that
			\begin{align*}
				\lambda(\langle \G U(t), U(t) \rangle) \in [C_1, C_2], \qquad \forall\, t \in [0,T),
			\end{align*}
			which completes the proof.
		\end{proof}

		\subsection{Proof of Lemma~\ref{lemma: sequence U converge}}\label{proof of lemma: sequence U converge}
		\begin{proof}
			Note that 
			\begin{align*}
				\int_0^{\infty} \|U'(t)\|_a^2 \ \text{d}t \leqslant - C \int_0^{\infty} E'(U(t)) \ \text{d}t  \leqslant  C E(U^0).
			\end{align*}
			This inequality implies the existence of a sequence $\{t^n\}$ such that $\|U'(t^n)\|_a \to 0$. 
			
			The boundedness of $U(t^n)$ guarantees that there exists $\bar{U} \in [H_0^1(\Omega)]^N$ such that $U(t^n) \rightharpoonup \bar{U}$ weakly in $[H_0^1(\Omega)]^N$. Moreover, strong convergence in $[L^2(\Omega)]^N$ shows that $\langle \bar{U}, \bar{U} \rangle = I_N$ and $\G U(t^n) \to \G \bar{U}$ strongly in $[H_0^1(\Omega)]^N$. Then we have 
			\begin{align*}
				- U(t^n) \langle \G U(t^n), U(t^n) \rangle + \G U(t^n) \langle  U(t^n), U(t^n) \rangle \rightharpoonup - \bar{U} \langle \G \bar{U}, \bar{U} \rangle + \G \bar{U} \langle  \bar{U}, \bar{U} \rangle
			\end{align*}
			weakly in $[H_0^1(\Omega)]^N$. This implies 
			\begin{align*}
				- \bar{U} \langle \G \bar{U}, \bar{U} \rangle + \G \bar{U} \langle  \bar{U}, \bar{U} \rangle = 0,
			\end{align*}
			which means that $\bar{U}$ is a solution of \eqref{equ: model problem 1} and, consequently, a critical point of the energy $E(\cdot)$. 
			
			Since $ \|U'(t^n) \langle \G U(t^n), U(t^n) \rangle^{-1} \|_a \to 0 $, it can be obtained that
			\begin{align*}
				(U(t^n), U'(t^n) \langle \G U(t^n), U(t^n) \rangle^{-1})_a \to 0 = (\bar{U}, P_{\bar{U}}\bar{U})_a.
			\end{align*}
			This convergence implies that $\|U(t^n)\|_a \to \|\bar{U}\|_a$, which, together with the weak convergence, completes the proof.
		\end{proof}

		\subsection{Proof of Lemma~\ref{lemma:convergence of equivalence}}\label{proof of lemma:convergence of equivalence}
		\begin{proof}
			We prove the conclusion by contradiction. Suppose there exists a $\{t^n\}$ such that
			\begin{align}
				\| [U(t^n)] - [U^*] \|_a \geqslant \eps, \quad \forall n \in \mathbb{N}
			\end{align}
			for some $\eps > 0$. Similar to Lemma \ref{lemma: sequence U converge}, there exists a subsequence $\{t^{n_k}\}$ and $\bar{U}\in [H_0^1(\Omega)]^N$ such that 
			\begin{equation*}
				\|U(t^{n_k}) -  \bar{U} \|_a \to 0 \qquad \text{as $n \to \infty$}.
			\end{equation*}
			
			Since the decay of energy and $\bar{U} \in \mathcal{M}^N$ is a solution of \eqref{equ: model problem 1}, we conclude that 
			\begin{align*}
				\lim\limits_{k \to \infty} E(U(t^{n_k})) = 	E(\bar{U})  = E_{\text{GS}},
			\end{align*}
			which means that $\bar{U}$ is a ground state, that is, $\bar{U} \in [U^*]$.
			
			Therefore,
			\begin{equation*}
				\begin{aligned}
					\| [U(t^{n_k})] - [U^*] \|_a =  \min\limits_{Q \in \mathcal{O}^N} \| U(t^{n_k}) - U^*Q\|_a
					\leqslant \| U(t^{n_k}) - \bar{U}\|_a \to 0 \quad \text{ as } k\rightarrow \infty,
				\end{aligned}
			\end{equation*}
			which contradicts the assumption that $\| [U(t^n)] - [U^*] \|_a \geqslant \eps$ for all $n \in \mathbb{N}$. 
		\end{proof}

		\subsection{Proof of Lemma~\ref{lemma: derivate of U(t) tends to 0}}\label{proof of lemma: derivate of U(t) tends to 0}
		
		\begin{proof}
			We see from Lemma \ref{lemma:convergence of equivalence} that for every $t \geqslant 0$, there exists a $Q(t)\in \mathcal{O}^N$ such that
			\begin{align*}
				\|U(t)Q(t) - U^*\|_a \to 0 \qquad \text{as $t \to \infty$},
			\end{align*}
			which yields
			\begin{align*}
				U'(t)Q(t) & = - U(t)Q(t) \langle \G U(t)Q(t), U(t)Q(t) \rangle + \G U(t)Q(t) \langle U(t)Q(t), U(t)Q(t) \rangle \\
				& \to - U^* \langle \G U^*, U^* \rangle + \G U^* \langle U^*, U^* \rangle = 0,
			\end{align*}
			as $t \to \infty$. Consequently, we have
			\begin{align*}
				\|U'(t)\|_a = \|U'(t)Q(t)\|_a \to 0
			\end{align*}
			and complete the proof.
		\end{proof}

		\subsection{Proof of Lemma~\ref{lem: converge to barU}}\label{proof of lem: converge to barU}
		
		\begin{proof}
			Note that 
			\begin{align*}
				\sum_{n = 0}^{\infty} \| \mathcal{L}_{U^n}U^{n}\|_a^2  \leqslant \sum_{n = 0}^{\infty} \frac{C}{\tau_n} (E(U^n) - E(U^{n+1}))  \leqslant \frac{C}{\tau_{\text{min}}} E(U^0) < \infty.
			\end{align*}
			This means the series $\sum_{n = 0}^{\infty} \| \mathcal{L}_{U^n}U^{n}\|_a^2$ converges, and thus
			\begin{align*}
				\lim_{n \to \infty} \| \mathcal{L}_{U^n}U^{n}\|_a^2 = 0.
			\end{align*}
			
			Suppose $U^n \rightharpoonup \bar{U}$ weakly in $[H_0^1(\Omega)]^N$. It follows that $U^n \to \bar{U}$ strongly in $[L^2(\Omega)]^N$. Hence, 
			\begin{align*}
				\langle \bar{U}, \bar{U} \rangle = \lim_{n \to \infty} \langle U^n, U^n \rangle = I_N.
			\end{align*}
			Moreover, we have $\langle \G U^n, U^n \rangle \to \langle \bar{U}, \bar{U} \rangle$. Consequently,
			\begin{align*}
				\mathcal{L}_{U^n}U^{n} \rightharpoonup \mathcal{L}_{\bar{U}}U^{*}
			\end{align*}
			weakly in $[H_0^1(\Omega)]^N$. Therefore, $\mathcal{L}_{\bar{U}}\bar{U}= 0$, which means $\bar{U}$ is a solution of (\ref{equ: model problem 0}). Observe that
			\begin{align*}
				\langle U^n, \mathcal{L}_{U^n}U^{n} \rangle_a \to 0 = \langle \bar{U}, \mathcal{L}_{\bar{U}}\bar{U}\rangle_a.
			\end{align*}
			This implies $\|U^n\|_a \to \|\bar{U}\|_a$, and hence the strong convergence follows.
		\end{proof}
		
		\subsection{Proof of Lemma~\ref{lem: same projection}}\label{proof of lem: same projection}
		
		\begin{proof}
			Since $\langle U^*, U^* \rangle = I_N$, there holds
			\begin{align*}
				\P U = U^*\langle U^*, U \rangle.
			\end{align*}
			Furthermore,
			\begin{align*}
				\P_a U & = U^*\langle U^*, U^* \rangle_a^{-1} \langle U^*, U \rangle_a \\
				& = U^* \Lambda^{-1} \Lambda \langle  U^*, U \rangle \\
				& = U^*\langle  U^*, U \rangle,
			\end{align*}
			which completes the proof.
		\end{proof}
		
		\subsection{Proof of Lemma~\ref{lemma: PGU=GPU}}\label{proof of lemma: PGU=GPU}
		\begin{proof}
			Since $\G U^{*} = U^* \Lambda^{-1}$,
			it follows that 
			\begin{align*}
				\P_{\bot} (\G U) & = \G U - U^*\langle U^*,  \G U \rangle \\
				& = \G U - U^* \Lambda^{-1} \langle U^*,   U \rangle
			\end{align*}
			and
			\begin{align*}
				\G (\P_{\bot}U) 
				& = \G U - \G U^* \langle U^*,   U \rangle \\
				& = \G U - U^* \Lambda^{-1} \langle U^*,   U \rangle.
			\end{align*}
			This completes the proof.
		\end{proof}
		
		\subsection{Proof of Lemma~\ref{lem: exponential convergence of numerical scheme}}\label{proof of lem: exponential convergence of numerical scheme}
		\begin{proof}
			Note that 
			\begin{align*}
				\|\P_{\bot}U^{n+1}\|_a^2 & = \|\P_{\bot}(U^n - \tau_n\L_{U^n}U^{n+\frac12})\|_a^2 \\
				& = \|\P_{\bot}U^n\|_a^2 - 2\tau_n (\P_{\bot}U^n, \P_{\bot}\L_{U^n}U^{n+\frac12})_a + \tau_n^2 \| \P_{\bot}\L_{U^n}U^{n+\frac12}\|_a^2.
			\end{align*}
			Since
			\begin{align*}
				\P_{\bot}\L_{U^n}U^{n+\frac12}& = \P_{\bot} \Big( U^n \langle \G U^n, U^{n+ \frac12} \rangle - \G U^n \langle U^n, U^{n+\frac12} \rangle \Big) \\
				& = \P_{\bot} U^n \langle \G U^n, U^{n+ \frac12} \rangle - \P_{\bot} \big(\G U^n \langle U^n, U^{n+\frac12} \rangle \big) \\
				& = \P_{\bot} U^n \langle \G U^n, U^{n+ \frac12} \rangle - \G \P_{\bot} U^n \langle U^n, U^{n+\frac12} \rangle,
			\end{align*}
			we have 
			\begin{align*}
				\| \P_{\bot}\L_{U^n}U^{n+\frac12} \|_a \leqslant C \|\P_{\bot}U^n\|_a
			\end{align*}
			and 
			\begin{align*}
				\| \P_{\bot}\L_{U^n}V \|_a \leqslant C \|\P_{\bot}U^n\|_a \|V\|, \qquad \forall V \in [H_0^1(\Omega)]^N.
			\end{align*}
			Therefore, we obtain
			\begin{align*}
				\|\P_{\bot}U^{n+1}\|_a^2 &\leqslant (1 + C \tau_n^2) \|\P_{\bot}U^n\|_a^2  - 2\tau_n (\P_{\bot}U^n, \P_{\bot}\L_{U^n}U^{n})_a \\
				& \quad - 2\tau_n (\P_{\bot}U^n, \P_{\bot}\L_{U^n}U^{n+\frac12} - \P_{\bot}\L_{U^n}U^{n})_a.
			\end{align*}
			For the last term, we may estimate as follows
			\begin{align*}
				& |(\P_{\bot}U^n, \P_{\bot}\L_{U^n}U^{n+\frac12} - \P_{\bot}\L_{U^n}U^{n+\frac12})_a| \\
				& \leqslant \|\P_{\bot}U^n\|_a \|\P_{\bot}\L_{U^n}U^{n+\frac12} - \P_{\bot}\L_{U^n}U^{n}\|_a \\
				& \leqslant C \|\P_{\bot}U^n\|_a^2 \|U^{n+1} - U^n\|_a \leqslant C \tau_n \|\P_{\bot}U^n\|_a^2.
			\end{align*}
			For the second term, there holds
			\begin{align*}
				\P_{\bot}\L_{U^n}U^{n}& = \P_{\bot} \Big( U^n\langle \G U^n, U^n \rangle  - \G U^n  \Big) \\
				& = \P_{\bot} U^n \langle \G U^n, U^n \rangle - \G \P_{\bot}U^n,
			\end{align*}
			which implies
			\begin{align*}
				(\P_{\bot}U^n, \P_{\bot}\L_{U^n}U^{n})_a = \text{tr} (\langle \P_{\bot}U^n, \P_{\bot}U^n \rangle_a \langle \G U^n, U^n \rangle) - \| \P_{\bot}U^n \|^2.
			\end{align*}
			Noting that $[U^n] \to [U^*]$, for any $\epsilon \in (0,\frac{1}{\lambda_N})$, there exists a $n_0\in \mathbb{N}+$ such that
			\begin{align*}
				\lambda_{\min} (\langle \G U^n, U^n \rangle)  \geqslant \frac{1}{\lambda_N} - \epsilon,  \quad \forall n\geqslant n_0.
			\end{align*}
			Consequently,
			\begin{align*}
				(\P_{\bot}U^n, \P_{\bot}\L_{U^n}U^{n})_a & \geqslant \left( \frac{1}{\lambda_N} - \epsilon \right) \| \P_{\bot}U^n \|_a^2 - \frac{1}{\lambda_{N+1}} \|\P_{\bot}U^n \|_a^2 \\
				& = \left( \frac{1}{\lambda_N} - \frac{1}{\lambda_{N+1}} - \epsilon \right) \|\P_{\bot}U^n \|_a^2.
			\end{align*}
			Combining the above estimates, we arrive at
			\begin{equation*}
				\begin{aligned}
					\|\mathcal{P}_{\bot}U^{n+1}\|_a^2 \leqslant \left(1 + C \tau_n^2 - 2 \left( \frac{1}{\lambda_N} - \frac{1}{\lambda_{N+1}} - \epsilon  \right)\tau_n\right)\|\mathcal{P}_{\bot}U^{n}\|_a^2.
				\end{aligned}
			\end{equation*}
			For any $ \epsilon \in (0,\frac{1}{\lambda_N} - \frac{1}{\lambda_{N+1}})$, denote $\omega = \sup\limits_{\tau \in [\tau_{\min}, \tau_{\max}]} \left(1 + C \tau^2 - 2 \left( \frac{1}{\lambda_N} - \frac{1}{\lambda_{N+1}} - \epsilon  \right)\tau\right)^{\frac12}$. It is obvious that there exists an interval $[\tau_{\min}, \tau_{\max}]$ such that $\omega \in (0,1)$. Thus, we conclude that
			\begin{align*}
				\|\P_{\bot}U^{n+1}\|_a \leqslant \omega \|\P_{\bot}U^n\|_a, \qquad \forall n \geqslant n_0,
			\end{align*}
			which completes the proof.
		\end{proof}

		\subsection{Proof of Lemma~\ref{lemma: [Un] tends to [U_real]}}\label{proof of lemma: [Un] tends to [U_real]}
		
		\begin{proof}
			Note that the function $f(\theta) = \frac{\sin\frac{\theta}{2}}{\sin\theta}$ is continuous on $[0,\frac{\pi}{2}]$, then
			\begin{align}
				\sin\frac{\theta}{2} \leqslant C \sin\theta, \qquad \forall\, \theta \in [0,\frac{\pi}{2}].
			\end{align}
			Consequently, we have
			\begin{align}
				\| [U^n] - [U^*]\|^2 = \sum_{j=1}^N 4\sin^2\frac{\theta_j}{2} \leqslant C \sin^2\theta_N = C \left( {\delta}_{L^2}(\text{span}(U^n),\text{span}(U^*)) \right)^2 \leqslant C e^{-2cn}.
			\end{align}
			
			Furthermore, by the definition of $\| [U^n] - [U^*]\| $ and $ \| [U^n] - [U^*]\|_a$, there exist $Q^n_1\in \mathcal{O}^N$ and $Q^n_2\in  \mathcal{O}^N$ depending on $U^n$, such that
			\begin{align}
				\| [U^n] - [U^*]\| = \|U^n - U^*Q^n_1\| \quad \text{and} \quad \| [U^n] - [U^*]\|_a = \|U^n - U^*Q^n_2\|_a.
			\end{align}
			Moreover, observe that
			\begin{align*}
				\| [U^n] - [U^*]\|_a^2 & = \|U^n - U^*Q^n_2\|_a^2 \\
				& = \|\P_{\bot}U^n\|_a^2 + \|\P U^n - U^*Q^n_2\|_a^2 \\
				& \leqslant \|\P_{\bot}U^n\|_a^2 + \|\P U^n - U^*Q^n_1\|_a^2.
			\end{align*}
			It follows from the equivalence of norms $\|\cdot\|$ and $\|\cdot\|_a$ on the finite dimension space $\text{span}(U^*)$ that
			\begin{equation*}
				\begin{aligned}
					\| [U^n] - [U^*]\|_a^2& \leqslant \|\P_{\bot}U^n\|_a^2 + C \|\P U^n - U^*Q^n_1\|^2 \\
					& \leqslant\|\P_{\bot}U^n\|_a^2 + C \|U^n - U^*Q^n_1\|^2,
				\end{aligned}
			\end{equation*}
			where the constant $C$ depends on the largest eigenvalue $\lambda_N$ and Lemma \ref{lem: same projection}. This inequality implies 
			\begin{align*}
				\| [U^n] - [U^*]\|_a^2 \leqslant C e^{-2cn},
			\end{align*}
			which completes the proof.
		\end{proof}

	\bibliographystyle{siamplain}
	\bibliography{references}
	
\end{document}